\documentclass[letterpaper,12pt]{article}
\usepackage{latexsym,amssymb,amsfonts,amsmath,amsthm,nccmath,enumitem,setspace,graphicx,float}
\usepackage[dvipsnames]{xcolor}
\usepackage[hidelinks]{hyperref}
\usepackage[margin=2.5cm]{geometry}
\usepackage{subcaption}
\usepackage{multicol}
\usepackage{dsfont}
\usepackage{bm}
\usepackage{float}
\usepackage{hyperref}
\usepackage{caption}
\usepackage{mathrsfs}

\usepackage[normalem]{ulem}

\newtheorem{theorem}{Theorem}[section]
\newtheorem{corollary}[theorem]{Corollary}
\newtheorem{lemma}[theorem]{Lemma}

\theoremstyle{definition}
\newtheorem{definition}[theorem]{Definition}

\newtheorem{example}[theorem]{Example}

\usepackage{tikz}
\tikzstyle{vertex}=[circle, draw=black, fill=black, minimum size=2pt, inner sep=2]
\usetikzlibrary{arrows, snakes,arrows.meta,calc,decorations.markings,math,arrows.meta}

\newcommand{\len}{\textrm{len}}
\DeclareMathOperator{\op}{OP}

\makeatletter

\tikzset{
  on each segment/.style={
    decorate,
    decoration={
      show path construction,
      moveto code={},
      lineto code={
        \path [#1]
        (\tikzinputsegmentfirst) -- (\tikzinputsegmentlast);
      },
      curveto code={
        \path [#1] (\tikzinputsegmentfirst)
        .. controls
        (\tikzinputsegmentsupporta) and (\tikzinputsegmentsupportb)
        ..
        (\tikzinputsegmentlast);
      },
      closepath code={
        \path [#1]
        (\tikzinputsegmentfirst) -- (\tikzinputsegmentlast);
      },
    },
  },
  mid arrow/.style={postaction={decorate,decoration={
        markings,
        mark=at position .5 with {\arrow[scale=1.2,#1]{stealth}}
      }}},
}

\usetikzlibrary{decorations.markings}
\tikzset{
    set arrow inside/.code={\pgfqkeys{/tikz/arrow inside}{#1}},
    set arrow inside={end/.initial=>, opt/.initial=},
    /pgf/decoration/Mark/.style={
        mark/.expanded=at position #1 with
        {
            \noexpand\arrow[\pgfkeysvalueof{/tikz/arrow inside/opt}]{\pgfkeysvalueof{/tikz/arrow inside/end}}
        }
    },
    arrow inside/.style 2 args={
        set arrow inside={#1},
        postaction={
            decorate,decoration={
                markings,Mark/.list={#2}
            }
        }
    },
}

\makeatother

\let\oldproofname=\proofname
\renewcommand{\proofname}{\rm\bf{\oldproofname}}

\usepackage{tikzsymbols}

\title{A complete solution to the directed Oberwolfach problem of order  $2 \pmod{4}$ with cycles of even lengths}
\author{A.~C.~Burgess\thanks{University of New Brunswick, Saint John, NB, E2L 4L5, Canada.  Email: andrea.burgess@unb.ca.}, P.~H.~Danziger\thanks{Toronto Metropolitan University, Toronto, ON, M5B 2K3, Canada. Email: danziger@torontomu.ca}, A.~Lacaze-Masmonteil\thanks{University of Regina, Regina, SK, S4S 0A2, Canada. Email: Email: alk004@uregina.ca.}}

\begin{document}

\maketitle

\begin{abstract}
The Oberwolfach problem asks for a $2$-factorization of the complete graph in which each $2$-factor is isomorphic to a specific factor $F$. Recently, this problem has been extended to directed graphs. In this case, the directed Oberwolfach problem asks for a directed 2-factorization of the complete symmetric digraph in which each directed $2$-factor is isomorphic to a specific directed factor $F$. In this paper, we consider the directed Oberwolfach problem with directed 2-factors comprised of cycles of even lengths. Specifically,  we provide a complete solution to this particular case when the order of the complete symmetric digraph is congruent to 2 modulo 4.

\end{abstract}



\section{Introduction}

The Oberwolfach problem was first posed by Ringel in 1967~\cite{Ringel}, who asked the following: 
\begin{quote}
    In a venue with round tables of sizes $m_1, m_2, \ldots, m_t$, where $m_1+m_2+\cdots+m_t=n$, is it possible to find a seating plan for $n$ people over $(n-1)/2$ successive meals such that each person sits next to each other person exactly once?
\end{quote}
The Oberwolfach problem can be modeled as a graph factorization problem; to describe how, we first introduce the requisite terminology.  

Let $G$ be a graph.  We say that $G$ is {\em decomposed} into subgraphs $H_1, H_2, \ldots, H_s$ if the $H_i$ are pairwise edge-disjoint and their edge sets partition $E(G)$.  If $H_1 \cong H_2 \cong \cdots \cong H_s \cong H$, then we speak of an $H$-decomposition of $G$.  A {\em $k$-factor} of $G$ is a spanning $k$-regular subgraph of $G$; we will be particularly interested in the case $k=2$. In this case, a $2$-factor may be thought of as a collection of disjoint cycles whose lengths sum to the order of $G$.  A {\em $2$-factorization} of $G$ is a decomposition of $G$ into $2$-factors.  For a given $2$-factor $F$ of $G$, we will refer to an $F$-decomposition of $G$ as an $F$-factorization of $G$. We denote the 2-regular graph consisting cycles of length $m_1, m_2,\ldots, m_t$ (directed or undirected) by $[m_1, m_2, \ldots, m_t]$.  We use the standard exponential notation $[m_1^{\alpha_1}, m_2^{\alpha_2}, \ldots, m_t^{\alpha_t}]$ to denote a  $2$-regular graph containing $\alpha_i$ cycles of length $m_i$ for each $i \in \{1, \ldots, t\}$.

Ringel's original formulation of the Oberwolfach problem asks whether $K_n$ admits an $F$-factorization, where $F$ is a given $2$-factor of order $n$. As a $2$-factorization of $K_n$ is only possible when $n$ is odd, for even $n$ it is common to instead seek a factorization of the cocktail party graph $K_n-I$, formed by removing the edges of a $1$-factor $I$ from a complete graph of order $n$.  We thus consider the Oberwolfach problem, denoted $\op(F)$, to ask whether $K_n$ (if $n$ is odd) or $K_n-I$ (if $n$ is even) admits an $F$-factorization.  Various instances of $\op(F)$ have been solved, notably when $F$ is uniform (i.e.\ all cycles have the same length)~\cite{AlspachHaggkvist, ASSW, HoffmanSchellenberg}, bipartite (i.e.\ all cycles have even length)~\cite{BryantDanziger, Haggkvist}, or has two components~\cite{Traetta}.  The Oberwolfach problem has been completely solved for orders $n \leq 100$~\cite{DezaEtAl, Meszka, SalassaEtAl}; in this range, there are exactly four $2$-factors for which no solution exists, specifically $[3^2]$, $[3^4]$, $[4,5]$, and $[3^2,5]$.

Bryant and Scharaschkin~\cite{BryantScharaschkin} were the first to give a complete solution to the Oberwolfach problem for infinitely many orders, namely an infinite family of primes congruent to $1$~(mod~$16$) as well as orders of the form $2p$ where $p \equiv 5$~(mod~$24$) is prime; this result was extended to orders of the form $2p$ where $p$ is a prime congruent to $5$~(mod~$8$) in~\cite{ABHMS}.  It has recently been shown by non-constructive methods that a solution to $\op(F)$ exists for sufficiently large orders~\cite{GJKKO}.  However, despite the fact that this result stipulates that the set of unsolved cases is finite, no bound on the size of this set is known; thus the Oberwolfach problem remains wide open in general.  For further background on the Oberwolfach and related problems, we refer the reader to the survey~\cite{BDT-Survey}. 

In this paper, we will consider a variant of the Oberwolfach problem known as the {\em directed Oberwolfach problem} with tables of length $m_1, m_2,\ldots, m_t$, denoted $\op^*([m_1, m_2, \ldots, m_t])$. Introduced in \cite{BurSaj}, this problem considers a similar scenario as the Oberwolfach problem with the modified constraint that each guest be seated to the right of every other guest precisely once. The directed Oberwolfach problem is part of a recent surge of interest in directed variants of well-known cycle decomposition problems. For example, researchers have began to undertake the study of the question of existence of solutions to the directed Hamilton-Waterloo problem \cite{DirHamWateven, DirHamWat} as well as the directed Oberwolfach problem for the complete symmetric equipartite directed graph \cite{FranSaj}.

To formulate the directed Oberwolfach problem in graph-theoretic terms, we will require the following definition. A \emph{directed 2-factor of directed graph $G$ with directed cycles of length $m_1, m_2,\ldots, m_t$}, denoted $F=[m_1, m_2, \ldots, m_t]$, is a spanning subdigraph of $G$ comprised of the vertex-disjoint union of directed cycles of lengths $m_1, m_2, \ldots, m_t$. In \cite{BurSaj}, Burgess and \v{S}ajna posed the directed Oberwolfach problem as the question of the existence of a decomposition of the complete symmetric directed graph $K_n^*$ into copies of $F$. Although Burgess and \v{S}ajna were the first to formulate the directed Oberwolfach problem as a cycle decomposition problem, they remarked that the directed Oberwolfach problem could also be viewed as a design theoretic problem by constructing a design known as a resolvable Mendelsohn design \cite{Handbook}. This particular approach is one taken by the authors of \cite{BenZha, BerGerSot,  ZanDu}.

When $n$ is odd, any solution to $\op([m_1, \ldots, m_t])$ yields a solution of $\op^*([m_1, \ldots, m_t])$ by simply directing each 2-factor in the 2-factorization of $K_n$ in both orientations. Therefore, emphasis is placed on the case $n$ is even for $\op^*([m_1, m_2, \ldots, m_t])$. The directed Oberwolfach problem has been completely resolved for tables of uniform lengths. The complete solution to this particular case is given in Theorem \ref{Uniform} below.

\begin{theorem}[\cite{Abel, AdaBry, BenZha, BerGerSot, BurFranSaj, BurSaj, Alice, Til}] \label{Uniform}
Let $m \geq 2$ be an integer.  There exists a solution to $\op^*([m^{\alpha}])$ if and only if $(m,\alpha) \notin \{(4,1),(6,1), (3,2)\}$.
\end{theorem}

Zhang and Du~\cite{ZanDu} provide a complete solution to $\op^*$ when the seating arrangements are comprised of $\alpha$ tables of length three and one table of length four or five. 

\begin{theorem}[\cite{ZanDu}]
Let $\ell \in \{4,5\}$. For every $\alpha \geq 1$, there exists a solution to $\op^*([3^{\alpha}, \ell])$. 
\end{theorem}

The directed Oberwolfach problem was also completely resolved for the case in which we have two tables, as given in Theorem~\ref{twotab}. 

\begin{theorem}[\cite{DanielAlice, KadriSajna}]\label{twotab}
Let $2 \leq m_1 \leq m_2$ be integers.  There exists a solution to $\op^*([m_1,m_2])$ if and only if $[m_1,m_2] \neq [3,3]$.
\end{theorem}

Kadri and Šajna \cite{KadriSajna} derived a recursive method that gave rise to several infinite families of solutions to the directed Oberwolfach problem. In addition, their results in conjunction with those of Quirion \cite{Qui}, completely settle the directed Oberwolfach problem for all $n \leqslant 17$. 

\begin{theorem}[\cite{KadriSajna, Qui}] \label{SmallOrders}
Let $m_1, \ldots, m_t$ be integers such that each $m_i \geq 2$ and $m_1 + \cdots +m_t \leq 17$.  There exists a solution to $\op^*([m_1, \ldots, m_t])$ if and only if $[m_1, \ldots, m_t] \notin \{[4], [6], [3,3]\}$.
\end{theorem}

In this paper, we will consider the case of $\op^*([m_1, m_2, \ldots, m_t])$ in which all directed cycles have even length, and solve this case when $n \equiv 2 \pmod{4}$.   In particular, our main result is the following theorem.

\begin{theorem}\label{thm:main}
    Let $F$ be a bipartite directed $2$-factor of order $n \equiv 2 \pmod{4}$.  There exists a solution to $\op^*(F)$ if and only if $F \neq [6]$.
\end{theorem}

We structure this paper as follows. In Section \ref{S:prem}, we provide the necessary notation and definitions required to describe our solutions. Then, in Section~\ref{S:red}, we show that it suffices to construct the desired decompositions for two particular classes of directed graphs that require four and nine directed $2$-factors, respectively, in their decomposition. In the case of the first class of directed graphs, a complete solution is also given in Section~\ref{S:red}. The purpose of Section~\ref{S:main} is to complete the proof of Theorem~\ref{thm:main} by further reducing the question of existence of a 2-factorization of the second class of directed graphs to the existence of a particular set of directed paths and cycles, which we then give. 

\section{Notation and preliminaries} \label{S:prem}

First, we introduce pertinent terminology and notation. Given a directed graph (digraph for short), we denote its vertex set as $V(D)$ and its arc set by $A(D)$. An arc from $x$ to $y$ is written as $xy$. If $D_1$ and $D_2$ are two arc-disjoint subdigraphs of $D$, we write $D_1\oplus D_2$ to refer to the subdigraph of $D$ that is union of $D_1$ with $D_2$. As noted, we denote a $2$-regular graph with cycles of lengths $m_1, m_2, \ldots, m_t$ by $[m_1, m_2, \ldots, m_t]$.  We use the standard exponential notation $[m_1^{\alpha_1}, m_2^{\alpha_2}, \ldots, m_t^{\alpha_t}]$ to denote a  $2$-regular graph containing $\alpha_i$ cycles of length $m_i$ for each $i \in \{1, \ldots, t\}$. Similarly, for digraphs, we denote by $[m_1^{\alpha_1}, \ldots, m_t^{\alpha_t}]$ a collection of vertex-disjoint directed cycles containing $\alpha_i$ cycles of length $m_i$ for each $i \in \{1, \ldots, t\}$.  We call a digraph of the form $[m_1^{\alpha_1}, \ldots, m_t^{\alpha_t}]$ a {\em $2$-regular digraph}.  Note that, in a $2$-regular digraph, every vertex either has in- and out-degree $1$ or in- and out-degree 0. While we use the same notation $[m_1^{\alpha_1}, \ldots, m_t^{\alpha_t}]$ in the undirected and directed cases, it should be clear from context whether we are referring to a directed or undirected 2-regular graph. Lastly, a {\em $2$-factor} of a graph $G$ is a spanning subgraph of $G$ that is $2$-regular and a {\em directed $2$-factor} of digraph $D$ is a spanning subdigraph $D$ that is a directed 2-regular digraph. 

The directed cycle $C=(x_0,x_1,\ldots, x_m)$ contains arcs $x_0x_1$, $x_1x_2, \ldots, x_{m-1}x_m, x_mx_0$.  If $C=(x_0,x_1,\ldots, x_m)$, then we denote by  $\overleftarrow{C}$ the directed cycle obtained by traversing the arcs of $C$ in the opposite direction, i.e.\ $\overleftarrow{C}=(x_m, x_{m-1}, \ldots, x_1, x_0)$.  We denote the \emph{directed path (or dipath for short}) $P$ containing arcs $x_0x_1, x_1 x_2, \ldots, x_{m-1}x_m$ by $\langle x_0, x_1, \ldots, x_m \rangle$. For the directed path $P=\langle x_0, x_1, \ldots, x_m \rangle$, we say that $x_0$ is the {\em source} of $P$, denoted by $s(P)$ and $x_m$ is the {\em terminal} of $P$, denoted $t(P)$. The set of \emph{internal vertices} of a dipath $P$ is the set $V(P)\setminus\{s(P), t(P)\}$.  The {\em length} of a dipath $P$ is its number of arcs, denoted $\len(P)$.  Given paths $P=\langle x_0, x_1, \ldots, x_m\rangle$ and $Q=\langle x_m, x_{m+1}, \ldots, x_{m+n} \rangle$, where $t(P)=s(Q)=x_m$, we define the {\em concatenation} of $P$ and $Q$ as the walk $P+Q = \langle x_0, x_1, \ldots, x_m, x_{m+1}, \ldots, x_{m+n} \rangle$.  Note that $P+Q$ is a path if $P$ and $Q$ share only vertex $x_m$, and is a cycle if we also have that $s(P)=t(Q)$ but $P$ and $Q$ have no common internal vertices.

Given a graph $G$, let $G^*$ be the digraph obtained by replacing each edge $\{x,y\}$ of $G$ with arcs $xy$ and $yx$.  Thus, the arc set of $G^*$ is the set $A(G^*) = \{xy, yx \mid \{x,y\} \in E(G)\}$.  In particular, $K_n^*$ is the complete symmetric digraph and $\op^*(F)$ asks for a directed $2$-factorization of $K_n^*$ in which each directed $2$-factor is isomorphic to $F$.  

While the directed Oberwolfach problem asks for a $2$-factorization of $K_n^*$, we will also consider $2$-factorizations of auxiliary digraphs of the form $G^*$ where $G$ is a lexicographic product of graphs, which we now define.

\begin{definition}
Let $G$ and $H$ be graphs.  The {\em lexicographic (wreath) product of $G$ with $H$}, denoted $G[H]$, is the graph with vertex set $V(G) \times V(H)$, where $(g_1,h_1)$ and $(g_2,h_2)$ are adjacent if and only if either $\{g_1,g_2\} \in E(G)$, or $g_1=g_2$ and $\{h_1,h_2\} \in E(H)$.  
\end{definition}

\noindent
Informally, $G[H]$ may be thought of as being formed by blowing up the vertices of $G$ into copies of $H$, and replacing each edge of $G$ with a complete bipartite graph with parts of size $|V(H)|$.  

\begin{definition}
Let $m>1$ be an integer and $S\subseteq \mathbb{Z}_m \setminus \{0\}$ such that $S$ is inverse closed. The {\em circulant graph} Circ$(m, S)$ is the graph with vertex set $\mathbb{Z}_m$ and edge set $\{\{i,j\} \mid i-j\in S\}$. 
\end{definition}

We now define the underlying graph of the two auxiliary digraphs of interest.

\begin{definition}
For a positive integer $m$, let $H_{2m}$ denote the graph $C_m[\overline{K_2}]$. Define $W_{2m}$ to be the graph $\mathrm{Circ}(m,\{\pm 1,\pm 2\}) [K_2]$.
\end{definition}

For the graphs $H_{2m}$ and $W_{2m}$, as well as $J_{2m}$ defined in Section~\ref{S:main},  we take take the blowup of vertex $i$ to be $x_i$ and $y_i$. Figures~\ref{H-graph} and \ref{W-graphs} illustrate  $H_{14}$ and $W_{14}$, respectively. In each case, vertices with the same labels are identified.

\begin{figure}[htb]
\begin{center}
\begin{subfigure}{1\textwidth}
\centering
\begin{tikzpicture}[x=1.5cm,y=1.5cm,scale=1]
\foreach \x in {0,...,6}{
    \draw (\x,0) -- (\x+1,0);
    \draw (\x,1) -- (\x+1,1);
    \draw (\x,0) -- (\x+1,1);
    \draw (\x,1) -- (\x+1,0);
}
\foreach \x in {0,...,6}{
    \draw[ball color=black](\x,1) circle (3pt) node[above, inner sep=8pt]{$x_{\x}$};
    \draw[ball color=black](\x,0) circle (3pt) node[below, inner sep=8pt]{$y_{\x}$};
}

\draw[ball color=gray](7,1) circle (3pt) node[above, inner sep=8pt]{\color{gray} $x_{0}$};
\draw[ball color=gray](7,0) circle (3pt) node[below, inner sep=8pt]{\color{gray} $y_{0}$};

\end{tikzpicture}
\caption{The graph $H_{14}$.}
\label{H-graph}
\end{subfigure}

\begin{subfigure}{1\textwidth}
\centering    
\begin{tikzpicture}[x=1.5cm,y=1.5cm,scale=1]
\foreach \x in {0,...,6}{
    \draw (\x,0) -- (\x,1);
}
\foreach \x in {0,...,6}{
    \draw (\x,0) -- (\x+1,0);
    \draw (\x,1) -- (\x+1,1);
    \draw (\x,0) -- (\x+1,1);
    \draw (\x,1) -- (\x+1,0);
    \draw (\x,0) -- (\x+2,1);
    \draw (\x,1) -- (\x+2,0);
    \draw (\x,0) .. controls (\x+1,-0.25) .. (\x+2,0);
    \draw (\x,1) .. controls (\x+1,1.25) .. (\x+2,1);
}
\foreach \x in {0,...,6}{
    \draw[ball color=black](\x,1) circle (3pt) node[above, inner sep=10pt]{$x_{\x}$};
    \draw[ball color=black](\x,0) circle (3pt) node[below, inner sep=10pt]{$y_{\x}$};
}

\draw[ball color=gray](7,1) circle (3pt) node[above, inner sep=10pt]{\color{gray} $x_{0}$};
\draw[ball color=gray](7,0) circle (3pt) node[below, inner sep=10pt]{\color{gray} $y_{0}$};
\draw[ball color=gray](8,1) circle (3pt) node[above, inner sep=8pt]{\color{gray} $x_{1}$};
\draw[ball color=gray](8,0) circle (3pt) node[below, inner sep=8pt]{\color{gray} $y_{1}$};

\end{tikzpicture}
\caption{The graph $W_{14}$.}
\label{W-graphs}
\end{subfigure}
\end{center}
\caption{The underlying graphs of $H^*_{14}$ and $W^*_{14}$. Vertices with the same label are identified. }
\end{figure}
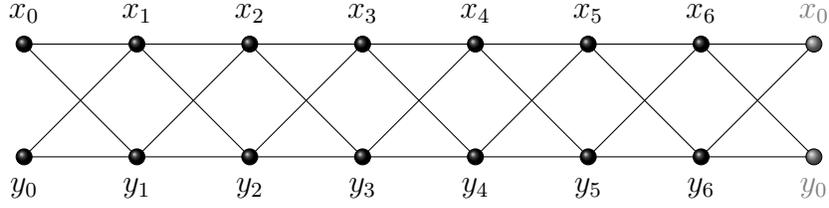
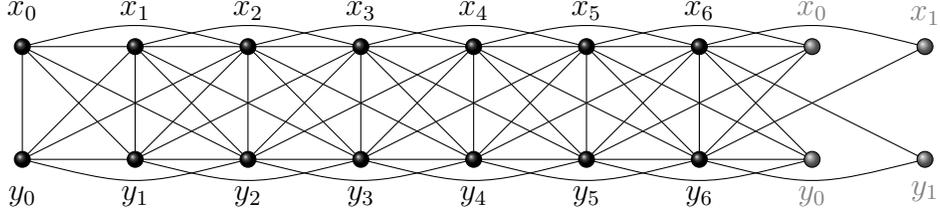

\section{Setup and approach} \label{S:red}

Our approach is inspired by the solution of the original undirected Oberwolfach problem for bipartite factors given in~\cite{BryantDanziger, Haggkvist}.  A similar approach was also recently successfully employed in~\cite{DanielAlice} to complete the solution of the two-table directed Oberwolfach problem.

Note that $\op^*(F)$ has been completely solved when $n<18$~\cite{KadriSajna, Qui}, so we may assume that $n \geq 17$, i.e.\ $m \geq 9$. In fact, our methods apply to all $m \geq 7$. The basic idea to solve $\op^*(F)$ in the case that $F$ is a directed $2$-factor of order $n=2m$, where $2m \equiv 2 \pmod{4}$, is as follows. First, we apply the following theorem to decompose $K_n^*$ into spanning subdigraphs isomorphic to $H_{n}^*$ and $W_{n}^*$. As a consequence of the proofs of Lemmas~10 and~13 of~\cite{DanielAlice}, Theorem \ref{WH} reduces the problem to finding $F$-factorizations of $H_n^*$ and $W_n^*$, which is the focus of this paper.  

\begin{theorem}[\cite{DanielAlice}] \label{WH}
Let $m \geq 7$ be and odd integer. There exists a decomposition of $K_{2m}^*$ into one copy of $W_{2m}^*$ and $\frac{m-5}{2}$ copies of $H_{2m}^*$.
\end{theorem}

To factor $H_{2m}^*$, we generalize a result of H\"{a}ggkvist~\cite{Haggkvist} to the directed case. In particular, H\"{a}ggkvist proved the following lemma.

\begin{lemma}[\cite{Haggkvist}] \label{HaggkvistLemma}
Let $m>1$ be an integer. For every bipartite $2$-factor $F$ of order $2m$, there exists an $F$-factorization of $H_{2m}$.
\end{lemma}

Lemma~\ref{HaggkvistLemma} immediately shows that, for any bipartite directed $2$-factor $F=[m_1,\ldots,m_t]$ not containing a directed $2$-cycles, the digraph $H_{2m}^*$ admits an $F$-factorization. Namely, one simply directs the cycles in an $[m_1, \ldots, m_t]$-factorization of $H_{2m}$ in both direction. The following lemma settles the case in which $F$ contains directed $2$-cycles. 

\begin{lemma} \label{Haggkvist_2-cycles}
Let $F=[m_1, m_2, \ldots, m_t]$ be a directed 2-factor such that $m_1 \leq m_2 \leq \ldots \leq m_t$, and suppose that $F$ has $s \geq 1$ directed cycles of length two. There exists an $F$-factorization of $H^*_{2m}$. 
\end{lemma}

\begin{proof}
From the hypothesis, it follows that $m_i=2$ for $i \in \{1, \ldots, s\}$ for some $s\geq 1$.  

If $s \geq 2$, by Lemma~\ref{HaggkvistLemma}, $H_{2m}$ admits a  $[{2s}, {m_{t-s}}, {m_{t-s+1}}, \ldots, {m_{t}}]$-factorization, $\cal H$. For each 2-factor in $\cal H$, we direct the cycles of size $m_i$, $i\in\{t-s,\ldots, t\}$, each way, giving two directed partial factors $\{F_1, F_2\}$. Now, take a 1-factorization, $\{I_1, I_2\}$, of the cycle of length $2s$. For each edge in $I_i$, $i\in \{1,2\}$, we replace it with the corresponding directed 2-cycle. This gives two sets of 2-cycles $I_1^*$ and $I_2^*$ that are both vertex-disjoint from $F_1$ and $F_2$. Now, for each $i \in \{1,2\}$, $F_i\cup I_i^*$ is a factor of $H_{2m}^*$ with the prescribed cycle structure. Repeating this for each 2-factor in $\cal H$ yields the required directed 2-factorization of $H^*_{2m}$.

We now consider the case $s=1$.  This means that $F=[2,m_2,m_3,\ldots,m_t]$. Note that we will require four directed 2-factors in our decomposition. Therefore, for each cycle length, we will construct four directed cycles.  First, we construct the following four directed cycles of length 2:

$$C_2^0=(x_0, x_{m-1});~~ C_2^1=(y_0, x_{m-1});~~C_2^2=(y_0, y_{m-1});~~C_2^3=(x_0, y_{m-1}). 
$$

\noindent For each $i \in \{2, \ldots, t\}$, let $m_i=2k_i$, where $k_i \geqslant 2$. We give a separate construction for the four directed cycles of length $m_2$.  

\noindent If $m_2 \equiv 0 \ (\textrm{mod}\ 4)$, then $k_2$ is even. Define
\[
    \begin{aligned} 
&C_{m_2}^0=(y_0, y_1, y_2, \ldots, y_{k_2}, x_{k_2-1}\, x_{k_2-2}, \ldots, x_4, x_3, x_2, x_1); \\
&C_{m_2}^1=(x_{0}, x_1, y_2, x_3, y_4, \ldots, x_{k_2-1}, y_{k_2}, y_{k_2-1}, x_{k_2-2}, y_{k_2-3}, x_{k_2-4}, y_{k_2-5}, \ldots, x_4, y_3, x_2, y_{1}); \\
&C_{m_2}^2=(x_0, y_1, x_2, x_3, x_4, \ldots, x_{k_2}, y_{k_2-1}, y_{k_2-2}, \ldots, y_4, y_3, y_2, x_1);\\ 
&C_{m_2}^3=(y_0, x_1, x_2, y_3, x_4,\ldots, y_{k_2-1}, x_{k_2}, x_{k_2-1}, y_{k_2-2}, x_{k_2-3}, y_{k_2-4}, x_{k_2-5}, \ldots, y_4, x_3, y_2, y_1). \\
     \end{aligned}
\]
If $m_2 \equiv 2 \ (\textrm{mod}\ 4)$, then $k_2$ is odd. Define
\[
    \begin{aligned} 
&C_{m_2}^0= (y_0, y_1, y_2, \ldots, y_{k_2}, x_{k_2-1}, x_{k_2-2}, \ldots, x_4, x_3, x_2, x_1); \\
&C_{m_2}^1=(x_{0}, x_1, x_2, x_3, x_4, \ldots, x_{k_2-1}, y_{k_2}, y_{k_2-1}, y_{k_2-2}, y_{k_2-3},\ldots, y_2, y_{1}); \\
&C_{m_2}^2=(x_0, y_1, x_2, y_3, x_4, y_5, \ldots,  x_{k_2-1}, x_{k_2}, y_{k_2-1}, x_{k_2-2}, y_{k_2-3}, x_{k_2-4}, y_{k_2-5}, \ldots, y_4, x_3, y_2, x_1);\\ 
&C_{m_2}^3=(y_0, x_1, y_2, x_3, y_4, x_5, \ldots, y_{k_2-1}, x_{k_2}, x_{k_2-1}, y_{k_2-2}, x_{k_2-3}, y_{k_2-4},x_{k_2-5}, \ldots, x_4, y_3, x_2, y_1). \\
     \end{aligned}
\]
Now let $i \in \{3, \ldots, t\}$ and $a=\sum_{j=1}^{i-1} k_j$. 

\noindent If $m_i \equiv 0 \ (\textrm{mod}\ 4)$, then $k_i$ is even. We define
\[
    \begin{aligned} 
&C_{m_i}^0=(x_a, x_{a+1}, x_{a+2}, \ldots, x_{a+k_i-1}, y_{a+k_i}, y_{a+k_i-1}, y_{a+k_i-2}, \ldots, y_{a+1}); \\
&C_{m_i}^2=(y_{a}, x_{a+1}, y_{a+2}, x_{a+3}, y_{a+4}, \ldots, x_{a+k_i-1}, x_{a+k_i}, y_{a+k_i-1}, x_{a+k_i-2}, y_{a+k_i-3},\ldots, y_{a+1}); \\
&C_{m_i}^1=\overleftarrow{C_{m_i}^0};\; C_{m_i}^3=\overleftarrow{C_{m_i}^2}.\\ 
\end{aligned}
\]
If $m_i \equiv 2 \ (\textrm{mod}\ 4)$, then $k_i$ is odd. We define
\[
    \begin{aligned} 
&C_{m_i}^0=(x_a, x_{a+1}, y_{a+2}, x_{a+3}, y_{a+4},x_{a+5}, \ldots, y_{a+k_i-1}, y_{a+k_i}, x_{a+k_i-1},y_{a+k_i-2},x_{a+k_i-3}, \ldots, y_{a+1}); \\
&C_{m_i}^2=(y_{a}, x_{a+1}, x_{a+2}, x_{a+3}, x_{a+4}, \ldots, x_{a+k_i}, y_{a+k_i-1}, y_{a+k_i-2}, y_{a+k_i-3},\ldots, y_{a+1}); \\
&C_{m_i}^1=\overleftarrow{C_{m_i}^0};\; C_{m_i}^3=\overleftarrow{C_{m_i}^2}.\\ 
     \end{aligned}
\]
Finally, for each $i \in \{0,1,2,3\}$ we form the following subdigraph
\[
F_i=\{C_2^i, C_{m_1}^i, C_{m_2}^i,\ldots, C_{m_t}^i\}.
\]
It is tedious but straightforward to verify that directed cycles in $F_i$ span the vertex set meaning that $F_i$ is an $F$-factor. Furthermore, $F_i$ and $F_j$ are arc-disjoint for all $i, j \in \{0,1,2,3\}$ and $i\neq j$. Consequently, the set $\{F_i \mid i=0,1,2,3\}$ is the desired $F$-factorization. \end{proof}

Combining Lemmas~\ref{HaggkvistLemma} and~\ref{Haggkvist_2-cycles}, we obtain the existence of $F$-factorizations of $H_{2m}^*$ for bipartite factors.

\begin{theorem}\label{DirectedHaggkvist}
Let $m>1$ be an integer and let $F$ be a bipartite directed $2$-factor of order $2m$.  There exists an $F$-factorization of $H_{2m}^*$.
\end{theorem}

In light of Theorems~\ref{WH} and~\ref{DirectedHaggkvist}, we can state the following theorem.

\begin{theorem} \label{W-reduction}
 Let $m \geq 7$ be an odd integer and let $F$ be a bipartite directed $2$-factor of order $2m$.  If there exists an $F$-factorization of $W_{2m}^*$, then there exists an $F$-factorization of $K_{2m}^*$.  
\end{theorem}

In conclusion, to solve $\op^*(F)$ for $F$ a bipartite directed $2$-factor, it suffices to construct an $F$-factorization of $W_{2m}^*$.

\section{Factorizations of $W_{2m}^*$} \label{S:main}

In this section, we construct $F$-factorizations of $W_{2m}^*$, where $m$ is odd and $F$ is a bipartite directed 2-factor.  To do this, we adapt and formalize a method used for the undirected case in \cite{BryantDanziger, Haggkvist}, the directed Oberwolfach problem in \cite{DanielAlice}, and the directed Hamilton-Waterloo problem in \cite{DirHamWateven}. The constructions given in this section also give rise to $F$-factorizations of $W_{2m}^*$ when $m$ is even. However, it is not known in this case whether these $F$-factorizations can be used to construct an $F$-factorization of $K^*_{2m}$.

To factor $W_{2m}^*$, it will be helpful to consider factorizations of an auxiliary digraph $J_{2m}^*$, which we now define by specifying the undirected graph $J_{2m}$.  

\begin{definition}
Let $J_{2m}$ to be the graph with vertex set $\{x_i, y_i \mid i \in \mathbb{Z}_{m+2}\}$ and edge set:
\begin{multline*}
    E(J_{2m}) = \{ \{x_i, y_i\} \mid 1 \leq i \leq m\} \cup  \{\{x_i, x_{i+1}\}, \{y_i, y_{i+1}\}, \{x_i, y_{i+1}\}, \{y_i, x_{i+1}\}, \{x_i, x_{i+2}\}, \\ \{y_i, y_{i+2}\}, \{x_i, y_{i+2}\}, \{y_i, x_{i+2}\} \mid  0 \leq i \leq m-1\}.
\end{multline*}
\end{definition}

\noindent Figure~\ref{Jgraph} illustrates $J_{14}$ as an example.  Note that $m=7$ in Figure~\ref{Jgraph}.

\begin{figure}[h]
\begin{center}
\begin{tikzpicture}[x=1.5cm,y=1.5cm,scale=1]
\foreach \x in {1,...,7}{
    \draw (\x,0) -- (\x,1);
}
\foreach \x in {0,...,6}{
    \draw (\x,0) -- (\x+1,0);
    \draw (\x,1) -- (\x+1,1);
    \draw (\x,0) -- (\x+1,1);
    \draw (\x,1) -- (\x+1,0);
    \draw (\x,0) -- (\x+2,1);
    \draw (\x,1) -- (\x+2,0);
    \draw (\x,0) .. controls (\x+1,-0.25) .. (\x+2,0);
    \draw (\x,1) .. controls (\x+1,1.25) .. (\x+2,1);
}
\foreach \x in {0,...,8}{
    \draw[ball color=black](\x,1) circle (3pt) node[above, inner sep=10pt]{$x_{\x}$};
    \draw[ball color=black](\x,0) circle (3pt) node[below, inner sep=10pt]{$y_{\x}$};
}
\end{tikzpicture}
\end{center}
\caption{The graph $J_{14}$.} \label{Jgraph}
\end{figure}
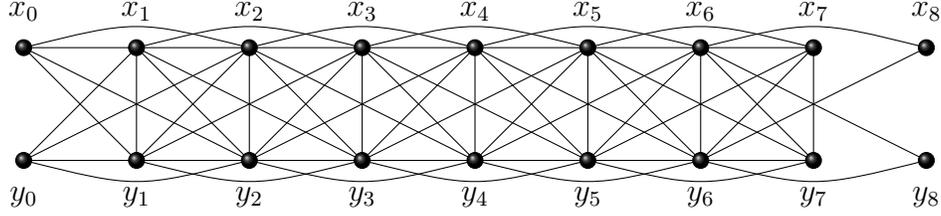

We may think of $J_{2m}^*$ as being formed by ``opening'' $W_{2m}^*$. Conversely, we can reverse this process by identifying vertices $x_i$ and $y_i$ of $J^*_{2m}$, where $i \in \{0,1\}$, with $x_{i+m}$ and $y_{i+m}$, respectively, to obtain $W^*_{2m}$.  

Note that a directed $2$-factor of $J_{2m}^*$ does not itself yield a directed $2$-factor of $W_{2m}^*$. However, a 2-regular digraph satisfying certain properties described in Definition \ref{defn:admiss} will in fact yield a $2$-factor of $W_{2m}^*$.

\begin{definition} \label{defn:admiss}
An {\em admissible} $2$-regular subdigraph of $J^*_{2m}$ is a $2$-regular subdigraph of $J^*_{2m}$ of order $2m$ that contains exactly one vertex from each of the sets $\{x_0,x_m\}$, $\{x_1,x_{m+1}\}$, $\{y_0,y_m\}$, $\{y_1,y_{m+1}\}$. An {\em admissible decomposition of $J^*_{2m}$} is a decomposition of $J^*_{2m}$ comprised of admissible 2-regular subdigraphs.
\end{definition}  

See Figures \ref{Fig:AdmiJ8} and \ref{Fig:AdmiJ6} for two examples of admissible $2$-regular subdigraphs of $J^*_{2m}$. Observe that an admissible $2$-regular subdigraph of $J_{2m}^*$ saturates all the vertices in $\{x_i,y_i \mid 2 \leq i \leq m-1\}$ and will yield a $2$-factor of $W_{2m}^*$ by identifying $x_m$ with $x_0$ and $y_m$ with $y_0$.  Furthermore, an admissible decomposition of $J^*_{2m}$ gives rise to a directed $2$-factorization of $W_{2m}^*$. We formalize this observation in Lemma \ref{JtoW} below.

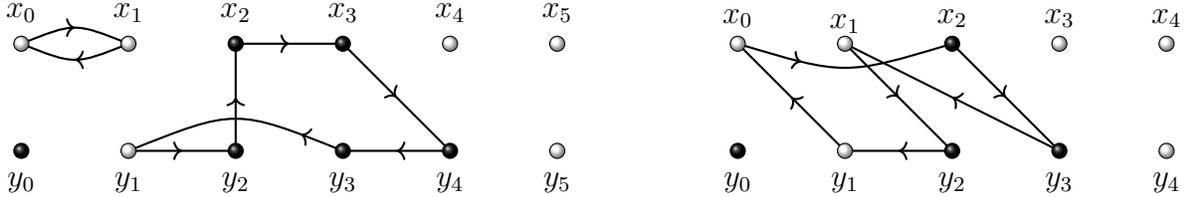
\begin{figure}[h]
\begin{subfigure}[b] {0.45 \textwidth}
\centering
\begin{tikzpicture}[x=1.5cm,y=1.5cm,scale=0.95]
\begin{scope}[thick,decoration={markings,mark=at position 0.5 with {\arrow{>}}}]
\draw[postaction=decorate] (0,1) .. controls (0.5,1.2) .. (1,1);
\draw[postaction=decorate] (1,1) .. controls (0.5,0.8) .. (0,1);
\draw[postaction=decorate] (1,0) -- (2,0);
\draw[postaction=decorate] (2,0) -- (2,1);
\draw[postaction=decorate] (2,1) -- (3,1);
\draw[postaction=decorate] (3,1) -- (4,0);
\draw[postaction=decorate] (4,0) -- (3,0);
\end{scope}
\begin{scope}[thick,decoration={markings,mark=at position 0.2 with {\arrow{>}}}]
\draw[postaction=decorate] (3,0) .. controls (2, 0.4) .. (1,0);
\end{scope}
\foreach \x in {2,3}{
    \draw[ball color=black](\x,1) circle (3pt) node[above, inner sep=8pt]{$x_{\x}$};
    }
\foreach \x in {4,5}{
    \draw[ball color=white](\x,1) circle (3pt)
    node[above, inner sep=8pt]{$x_{\x}$};
}
\foreach \x in {0,2,3}{
    \draw[ball color=black](\x,0) circle (3pt) node[below, inner sep=8pt]{$y_{\x}$};
}
\draw[ball color=white](5,0) circle (3pt) node[below, inner sep=8pt]{$y_{5}$};
\draw[ball color=white](0,1) circle (3pt)node[above, inner sep=8pt]{$x_{0}$};
\draw[ball color=white](1,1) circle (3pt)node[above, inner sep=8pt]{$x_{1}$};
\draw[ball color=white](1,0) circle (3pt)node[below, inner sep=8pt]{$y_{1}$};
\draw[ball color=black](4,0) circle (3pt)node[below, inner sep=8pt]{$y_{4}$};
\end{tikzpicture}
\caption{Admissible $2$-regular subdigraph of $J^*_8$.}
\label{Fig:AdmiJ8}
\end{subfigure}
\hfill
\begin{subfigure}[b]{0.45 \textwidth}
\centering
\begin{tikzpicture}[x=1.5cm,y=1.5cm,scale=0.95]
\begin{scope}[thick,decoration={markings,mark=at position 0.5 with {\arrow{>}}}]
\draw[postaction=decorate] (2,1) -- (3,0);
\draw[postaction=decorate] (3,0) -- (1,1);
\draw[postaction=decorate] (1,1) -- (2,0);
\draw[postaction=decorate] (2,0) -- (1,0);
\draw[postaction=decorate] (1,0) -- (0,1);
\end{scope}
\begin{scope}[thick,decoration={markings,mark=at position 0.3 with {\arrow{>}}}]
\draw[postaction=decorate] (0,1) .. controls (1,0.7) .. (2,1);
\end{scope}
\draw[ball color=black](2,1) circle (3pt)node[above, inner sep=6pt]{$x_{2}$};
\foreach \x in {3,4}{
    \draw[ball color=white](\x,1) circle (3pt) node[above, inner sep=6pt]{$x_{\x}$};
    }
\foreach \x in {0,2}{
    \draw[ball color=black](\x,0) circle (3pt) node[below, inner sep=8pt]{$y_{\x}$};
}
\draw[ball color=white](0,1) circle (3pt)node[above, inner sep=6pt]{$x_{0}$};
\draw[ball color=white](1,1) circle (3pt)node[above, inner sep=4pt]{$x_{1}$};
\draw[ball color=white](1,0) circle (3pt)node[below, inner sep=8pt]{$y_{1}$};
\draw[ball color=black](3,0) circle (3pt)node[below, inner sep=8pt]{$y_{3}$};
\draw[ball color=white](4,0) circle (3pt)node[below, inner sep=8pt]{$y_{4}$};

\end{tikzpicture}
\caption{Admissible $2$-regular subdigraph of $J^*_6$.}
\label{Fig:AdmiJ6}
\end{subfigure}
\caption{Admissible $2$-regular subdigraphs in $J^*_8$ and $J^*_6$ that saturate the same vertices in $\{x_0, x_1, y_0, y_1\}$.}
\label{Fig:PatternEx}
\end{figure}

\begin{lemma} \label{JtoW}
Let $F=[m_1,m_2,\ldots,m_t]$ be a $2$-regular digraph of order $2m$.  If there exists an admissible $F$-decomposition  of $J_{2m}^*$, then there exists an $F$-factorization of $W_{2m}^*$.
\end{lemma}

\begin{proof}
Let $\phi: V(J^*_{2m}) \rightarrow V(W^*_{2m})$, with $\phi(x_i)=x_j$ and $\phi(y_i)=y_j$, where $j \equiv i \ \pmod{m}$. That is, for $i \in \{0,1\}$, $\phi$ identifies the vertices of $J^*_{2m}$, $x_i$ and $y_i$, with $x_{i+m}$ and $y_{i+m}$, but leaves the rest of the vertices unchanged.

Let $\{F_1, F_2, \ldots, F_9\}$ be an admissible $F$-decomposition of $J_{2m}^*$. Given $D \in \{F_1, \ldots, F_9\}$, we construct a $2$-regular digraph $D'=\phi(D)$ in $W_{2m}^*$ by applying $\phi$ to the vertices of the directed cycles of $D$. It is straightforward to see that arcs of $D'$ are indeed arcs of $W^*_{2m}$.  Since $D$ has order $2m$ and contains only one vertex in each of the sets $\{x_0,x_m\}$, $\{x_1,x_{m+1}\}$, $\{y_0,y_m\}$ and $\{y_1,y_{m+1}\}$, this implies that $D'$ is a directed $2$-factor of $W_{2m}^*$.
    
Now let $F_i$ and $F_j$ be distinct elements of $\{F_1, \ldots, F_9\}$ and let $F_i'$ and $F_j'$ be the directed $2$-factors of $W_{2m}^*$ obtained by applying $\phi$ as described above.  Since the arcs of $W_{2m}^*$ are in one-to-one correspondence with the arcs of $J_{2m}^*$, the fact that $F_i$ and $F_j$ are arc-disjoint implies that $F_i'$ and $F_j'$ are also arc-disjoint.  This correspondence also shows that each arc of $W_{2m}^*$ appears in exactly one digraph in $\{F'_1, \ldots, F'_9\}$. Thus, we obtain an $F$-factorization of $W^*_{2m}$. \end{proof}

\begin{definition}
Let $F$ be an admissible $2$-regular subdigraph of $J^*_{2m}$.  The {\em external  pattern} of $F$ is the set $\{x_0, x_1, y_0, y_1\}\cap V(F)$. If $\mathcal{F}=\{F_1,  \ldots, F_9\}$ is an admissible decomposition of $J^*_{2m}$ into $2$-regular subdigraphs, and $X_i$ is the external pattern of $F_i$, then the list $(X_1, \ldots, X_9)$ is the {\em external pattern of $\mathcal{F}$.}  
\end{definition}

Observe that both $2$-regular digraphs in Figure~\ref{Fig:PatternEx} have external pattern $\{x_0, x_1, y_1\}$.  Of particular interest are pairs of $2$-regular admissible subdigraphs with the same external patterns. These pairs of subdigraphs are defined below. 

\begin{definition}
Let $F=[m_1,\ldots,m_s]$ and $F'=[m_1',\ldots,m_t']$ be $2$-regular admissible subdigraphs of $J^*_{2m}$ and $J^*_{2m'}$, with external patterns $X$ and $X'$, respectively.  If $X=X'$, then $F$ and $F'$ are {\em compatible}. Decompositions $\mathcal{F} =\{F_1, \ldots, F_9\}$ of $J^*_{2m}$ and $\mathcal{F}'=\{F_1', \ldots, F'_9\}'$ of $J^*_{2m'}$ into admissible $2$-regular digraphs are said to be {\em compatible} if, for each $i \in \{1, \ldots, 9\}$, $F_i$ and $F_i'$ are compatible.  
\end{definition}

Our next objective is to show that pairs of compatible decompositions can be spliced together to form a compatible decomposition of bigger order. To that end, we will find it useful to shift subgraphs of $J_{2m}^*$ by using the following function.

\begin{definition}
Let $\sigma$ be the function on $\{x_i,y_i \mid i \in \mathbb{Z}\}$ defined by $\sigma(x_i)=x_{i+1}$ and $\sigma(y_i)=y_{i+1}$. For a digraph $D$ on vertex set $\{x_i,y_i \mid 0 \leq i \leq m\}$, $\sigma(D)$ is the digraph obtained by applying $\sigma$ to each of its vertices (and correspondingly the arcs).  
\end{definition}

Note that $\sigma(D)$ is isomorphic to $D$, but $\sigma$ is not an automorphism; rather, we think of $\sigma(D)$ as a shift of $D$. Thus, for example, if $D$ is the $6$-cycle $(x_0, x_2, y_3, x_1, y_2, y_1)$ illustrated in Figure~\ref{Fig:AdmiJ6}, then $\sigma(D)$ is the $6$-cycle $(x_1, x_3, y_4, x_2, y_3, y_2)$.

In Lemma \ref{SpliceDecomps}, we show how the function $\sigma$ is used to construct an admissible decomposition of $J^*_{2m}$ from a pair of compatible admissible decompositions of $J_{2\ell_1}^*$ and $J_{2\ell_2}^*$, where $m=\ell_1+\ell_2$.  

\begin{lemma} \label{SpliceDecomps}
Let $m=\ell_1+\ell_2$, and let $\mathcal{F}_1$ be an admissible $[m_1, \ldots, m_s]$-decomposition of $J^*_{2\ell_1}$ and $\mathcal{F}_2$ be an admissible $[r_1,\ldots, r_t]$-decomposition of $J^*_{2\ell_2}$.  If $\mathcal{F}_1$ and $\mathcal{F}_2$ are compatible with external pattern $(X_1,\ldots, X_9)$, then there exists an admissible $[m_1,\ldots,m_s,r_1,\ldots,r_t]$-decomposition of $J^*_{2m}$ with external pattern $(X_1,\ldots, X_9)$.
\end{lemma}

\begin{proof}
Let $\mathcal{F}_1=\{F_1, \ldots, F_9\}$ and $\mathcal{F}_2=\{F'_1, \ldots, F'_9\}$ be two compatible decompositions of $J_{2\ell_1}^*$ and $J_{2\ell_2}^*$, respectively. Thus, $F_j$ and $F_j'$ are compatible for each $j\in \{1,\ldots, 9\}$. This means that, for $i \in \{0,1\}$, $x_{\ell_1+i}$ (resp.\ $y_{\ell_1+i}$) is a vertex of $F_j$ if and only if $x_i$ (resp.\ $y_i$) is not a vertex of $F_j'$.  Hence $F_j \oplus \sigma^{\ell_1}(F_j')$ is an admissible $2$-regular subdigraph of $J_{2(\ell_1+\ell_2)}^*$ with the same external pattern as $F_j$ (and hence $F'_j$). 
Therefore, $\mathcal{F} = \{F_j \oplus \sigma^{\ell_1}(F_j')\mid 1\leq j\leq 9\}$ is an admissible decomposition of $J^*_{2m}$ with external pattern $(X_1,\ldots, X_9)$.
\end{proof}

We illustrate an application of Lemma \ref{SpliceDecomps} in the following example. 

\begin{example} \label{ex:comp}
The admissible $2$-regular subdigraphs $F=[2,6]$ of $J_{8}^*$ and $F'=[6]$ of $J_{6}^*$ in Figures~\ref{Fig:AdmiJ8} and~\ref{Fig:AdmiJ6} are compatible. Therefore, by the construction in the proof of Lemma \ref{SpliceDecomps}, $F \oplus \sigma^4(F')$ is an admissible $2$-regular subdigraph $F''=[2,6,6]$ of $J^*_{14}$ with external pattern $X = \{x_0,x_1,y_1\}$.  This digraph is  illustrated in Figure~\ref{Fig:CompatibleEx}.  Note that $F' \oplus \sigma^3(F)$ is also an admissible $2$-regular subdigraph with the same external pattern~$X$. This concludes the example. 
\end{example}

\begin{figure}[h!]
\centering
\begin{tikzpicture}[x=1.5cm,y=1.5cm,scale=0.8]

\begin{scope}[thick,decoration={markings,mark=at position 0.5 with {\arrow{>}}}]
\draw[postaction=decorate] (0,1) .. controls (0.5,1.2) .. (1,1);
\draw[postaction=decorate] (1,1) .. controls (0.5,0.8) .. (0,1);
\draw[postaction=decorate] (1,0) -- (2,0);
\draw[postaction=decorate] (2,0) -- (2,1);
\draw[postaction=decorate] (2,1) -- (3,1);
\draw[postaction=decorate] (3,1) -- (4,0);
\draw[postaction=decorate] (4,0) -- (3,0);
\end{scope}
\begin{scope}[thick,decoration={markings,mark=at position 0.2 with {\arrow{>}}}]
\draw[postaction=decorate] (3,0) .. controls (2, 0.35) .. (1,0);
\end{scope}
\foreach \x in {2,...,5}{
    \draw[ball color=black](\x,1) circle (3pt) node[above, inner sep=8pt]{$x_{\x}$};
    }
\foreach \x in {0,2,3,5}{
    \draw[ball color=black](\x,0) circle (3pt) node[below, inner sep=8pt]{$y_{\x}$};
}
\draw[ball color=white](0,1) circle (3pt)node[above, inner sep=8pt]{$x_{0}$};
\draw[ball color=white](1,1) circle (3pt)node[above, inner sep=8pt]{$x_{1}$};
\draw[ball color=white](1,0) circle (3pt)node[below, inner sep=8pt]{$y_{1}$};
\draw[ball color=black](4,0) circle (3pt)node[below, inner sep=8pt]{$y_{4}$};

\begin{scope}[shift={(4,0)}]
\begin{scope}[thick,decoration={markings,mark=at position 0.5 with {\arrow{>}}}]
\draw[postaction=decorate] (2,1) -- (3,0);
\draw[postaction=decorate] (3,0) -- (1,1);
\draw[postaction=decorate] (1,1) -- (2,0);
\draw[postaction=decorate] (2,0) -- (1,0);
\draw[postaction=decorate] (1,0) -- (0,1);
\end{scope}
\begin{scope}[thick,decoration={markings,mark=at position 0.2 with {\arrow{>}}}]
\draw[postaction=decorate] (0,1) .. controls (1,0.8) .. (2,1);
\end{scope}
\draw[ball color=black] (2,1) circle (3pt) node[above, inner sep=6pt]{$x_6$};
\foreach \x in {7,8}{
    \draw[ball color=white](\x-4,1) circle (3pt) node[above, inner sep=6pt]{$x_{\x}$};
    }
\foreach \x in {6}{
    \draw[ball color=black](\x-4,0) circle (3pt) node[below, inner sep=8pt]{$y_{\x}$};
}
\draw[ball color=white](4,0) circle (3pt) node[below, inner sep=8pt]{$y_8$};
\draw[ball color=white](0,1) circle (3pt);
\draw[ball color=white](1,1) circle (3pt);
\draw[ball color=white](1,0) circle (3pt);
\draw[ball color=black](3,0) circle (3pt)node[below, inner sep=8pt]{$y_{7}$};
\end{scope}
\end{tikzpicture}
\caption{An admissible $[2,6,6]$-digraph of $J^*_{14}$ formed from the pair of admissible decompositions given in Figures~\ref{Fig:AdmiJ8} and~\ref{Fig:AdmiJ6} as described in Example \ref{ex:comp}.}
\label{Fig:CompatibleEx}
\end{figure}
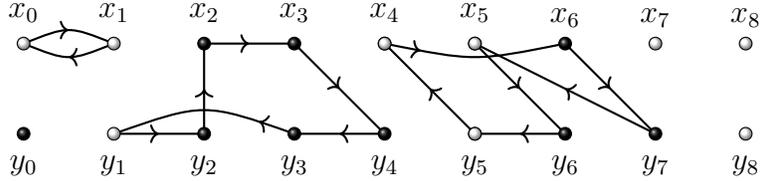

We will construct admissible $[m_1, \ldots, m_s]$-decompositions of $J_{2m}^*$ into $2$-regular digraphs with a small number of components, and then use  Lemma~\ref{SpliceDecomps} to splice together these decompositions. Our decompositions will all have the same external pattern $\mathcal{X}$, given below, and will thus be compatible:

\begin{multline*}
\mathcal{X}=(
\{y_1\}, \{x_0,x_1,y_1\}, \{x_1,y_0,y_1\}, \{x_0,x_1,y_0,y_1\}, \{y_1\}, \{x_0,x_1,y_0\}, \{x_1,y_1\}, \\ \{x_1\}, \{x_0,x_1,y_0,y_1\}
).
\end{multline*}

The $s$-tuples $[m_1, \ldots, m_s]$ are chosen to ensure that we can obtain all possible lists of even positive integers by combining them. 
This approach  greatly reduces the list of factorizations of $J_{2m}^*$ that we explicitly need to construct. See Appendix~\ref{[4,8]} for an example of an admissible decomposition of $J'_{8}$ with external pattern $\mathcal{X}$.

Before proceeding to our main constructions, we start with an example to illustrate the idea of how to construct a directed 2-factor of  $J_{2m}^*$ comprised of one long cycle of length $d+8k$, where $k \in \mathds{Z}$ and $d\in \{0, 2,4,6\}$, and several smaller cycles of fixed length.  The main goal is to reduce the construction of this 2-factor to the construction of four directed paths of small lengths satisfying particular properties and a (possibly empty) set of at most three cycles of lengths two or four.

\begin{example} \label{ex:capsplice}

Suppose we wish to build a $2$-regular digraph $D=[2m]$, where $2m \equiv 2 \pmod{8}$ and $2m \geq 10$, with external  pattern $\{x_0, x_1, y_1\}$.  Consider the digraphs $L$ and $R$ illustrated in Figures~\ref{fig:dipathL} and~\ref{fig:dipathR} .

\begin{figure}[ht]
\begin{subfigure} {0.45 \textwidth}
\centering
\begin{tikzpicture}[x=1.5cm,y=1.5cm,scale=0.75]
\foreach \x in {0,...,3} {
\coordinate (x\x) at (\x,1) {};
\coordinate (y\x) at (\x,0) {};
}
\begin{scope}[thick, decoration={markings, mark=at position 0.33 with {\arrow{>}}}]
\draw[postaction={decorate}] (y2) -- (x0);
\draw[postaction={decorate}] (x0) -- (y1);
\draw[postaction={decorate}] (y1) -- (x1);
\end{scope}
\begin{scope}[thick, decoration={markings, mark=at position 0.5 with {\arrow{>}}}]
\draw[postaction={decorate}] (x1) .. controls (2,0.7)  .. (x3);
\end{scope}
\foreach \x in {0,2,3}{
	\draw[ball color=black] (y\x) circle (3pt);
	\node[below=4pt] () at (y\x) {$y_{\x}$};
}
\draw[ball color=white] (y1) circle (3pt);
	\node[below=4pt] () at (y1) {$y_{1}$};
    
\foreach \x in {2,3}{
	\draw[ball color=black] (x\x) circle (3pt);
	\node[above=4pt] () at (x\x) {$x_{\x}$};
}
\foreach \x in {0,1}{
	\draw[ball color=white] (x\x) circle (3pt);
	\node[above=4pt] () at (x\x) {$x_{\x}$};
}
\draw[ball color=green] (x3) circle (3pt);
\draw[ball color=red] (y2) circle (3pt);
\end{tikzpicture}
\caption{Directed path of $J^*_4$, $L$.}
\label{fig:dipathL}
\end{subfigure}
\hfill
\begin{subfigure} {0.45 \textwidth}
\centering
\begin{tikzpicture}[x=1.5cm,y=1.5cm,scale=0.75]
    \begin{scope}
\foreach \x in {0,...,4} {
\coordinate (x\x) at (\x,1) {};
\coordinate (y\x) at (\x,0) {};
}

\begin{scope}[thick, decoration={markings, mark=at position 0.33 with {\arrow{>}}}]
\draw[postaction={decorate}] (y3) .. controls (2,0.4) .. (y1);
\end{scope}
\begin{scope}[thick, decoration={markings, mark=at position 0.33 with {\arrow{>}}}]
\draw[postaction={decorate}] (x0) .. controls (1,0.7) .. (x2);
\end{scope}
\begin{scope}[thick, decoration={markings, mark=at position 0.5 with {\arrow{>}}}]
\draw[postaction={decorate}] (x2) -- (y0);
\draw[postaction={decorate}] (x1) -- (y2);
\draw[postaction={decorate}] (y2) -- (y3);
\draw[postaction={decorate}] (y1) -- (x0);
\end{scope}

\foreach \x in {0,...,3}{
	\draw[ball color=black] (y\x) circle (3pt);
	\node[below=4pt] () at (y\x) {$y_{\x}$};
}
\foreach \x in {0,1,2}{
	\draw[ball color=black] (x\x) circle (3pt);
        \node[above=4pt] () at (x\x) {$x_{\x}$};
}
\draw[ball color=green] (x1) circle (3pt);
\draw[ball color=red] (y0) circle (3pt);

\foreach \x in {3,4}{
\draw[ball color=white] (x\x) circle (3pt);
 \node[above=4pt] () at (x\x) {$x_{\x}$};
}
\draw[ball color=white] (y4) circle (3pt);
\node[below=4pt] () at (y4) {$y_{4}$};
\end{scope}
\end{tikzpicture}
\caption{Directed path of $J^*_6$, $R$.}
\label{fig:dipathR}
\end{subfigure}
\\

\begin{subfigure}{1 \textwidth}
\centering
\begin{tikzpicture}[x=1.5cm,y=1.5cm,scale=0.75]
\foreach \x in {0,...,5} {
\coordinate (x\x) at (\x,1) {};
\coordinate (y\x) at (\x,0) {};
}
\begin{scope}[thick, decoration={markings, mark=at position 0.33 with {\arrow{>}}}]
\draw[postaction={decorate}, gray] (x1) -- (x0);
\draw[postaction={decorate},  gray] (x0) .. controls (1,0.7) .. (x2);
\draw[postaction={decorate},  gray] (x2) -- (x3);
\draw[postaction={decorate},  gray] (x3) .. controls (4,0.7) .. (x5);

\draw[postaction={decorate}] (y4) -- (y3);
\draw[postaction={decorate}] (y3) .. controls (2,0.4) .. (y1);
\draw[postaction={decorate}] (y1) -- (y2);
\end{scope}
\begin{scope}[thick, decoration={markings, mark=at position 0.67 with {\arrow{>}}}]
\draw[postaction={decorate}] (y2) .. controls (1,0.4) .. (y0);
\end{scope}
\foreach \x in {0,...,5}{
	\draw[ball color=black] (y\x) circle (3pt);
}
\foreach \x in {2,3}{
	\node[below=4pt] () at (y\x) {$y_{\x}$};
}

\foreach \x in {0,...,5}{
	\draw[ball color=black] (x\x) circle (3pt);
}

\foreach \x in {2,3}{
	\node[above=4pt] () at (x\x) {$x_{\x}$};
}
\draw (0,1) node[above=4pt] {$x_0$};
\draw (1,1) node[above=4pt] {$x_1$};
\draw (4,1) node[above=4pt] {$x_4$};
\draw (5,1) node[above=4pt] {$x_5$};
\draw (0,0) node[below=4pt] {$y_0$};
\draw (1,0) node[below=4pt] {$y_1$};
\draw (4,0) node[below=4pt] {$y_4$};
\draw (5,0) node[below=4pt] {$y_5$};
\draw[ball color=green] (x1) circle (3pt);
\draw[ball color=green] (x5) circle (3pt);
\draw[ball color=red] (y0) circle (3pt);
\draw[ball color=red] (y4) circle (3pt);
\end{tikzpicture}
\caption{Two directed paths, $Q$ (grey) and $U$ (black), such that $C=(Q, U)$. }
\label{fig:dipathsC}
\end{subfigure}
\caption{Four directed paths used to construct an admissible $[2m]$-factor of $J^*_{2m}$, as described in Example \ref{ex:capsplice}.}
\label{LeftRight}
\end{figure}
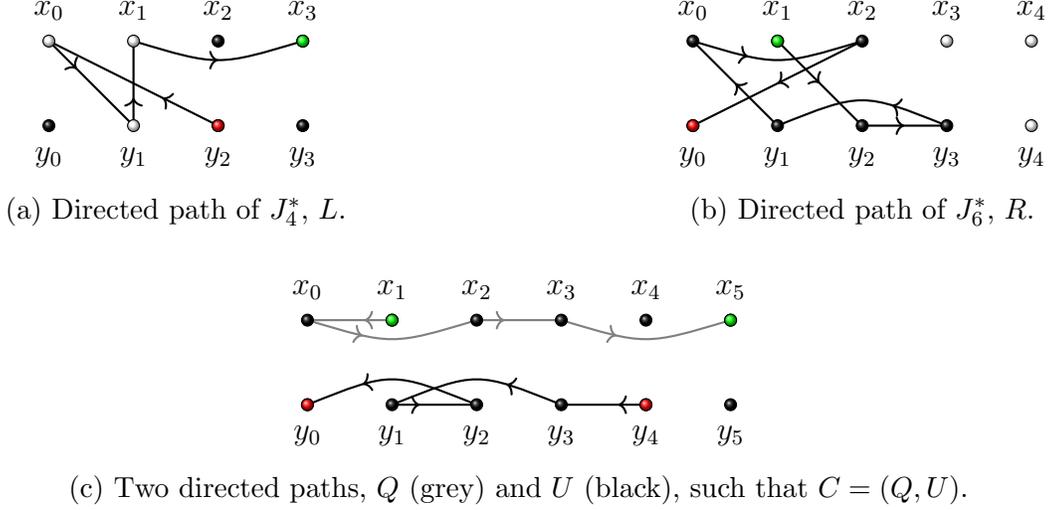

Observe that $L$ and $R$ are directed paths in $J_4^*$ and $J_6^*$.  Their lengths sum to $10 \equiv 2 \pmod{8}$. Moreover, $L$ saturates vertices $x_0$, $x_1$ and $y_1$, but not $y_0$, while $R$ saturates $y_3$ but not $x_3,x_4$ or $y_4$.  Lastly,  $L$ and $\sigma^2(R)$ may be concatenated to form a directed cycle of length $10$ of $J^*_{10}$, $L+\sigma^2(R)$. 

To form a cycle of length $2m=2(5+4k)=10+8k$ of $J^*_{2m}$, we place $k$ translations of the digraph $C=(Q,U)$, where $Q$ and $U$ are directed paths of $J_8$ illustrated in Figure \ref{fig:dipathsC}, between $L$ and $R$. It is not difficult to see that the concatenation 

\[
L+\sigma^2(Q)+\sigma^{6}(Q)+
\cdots+\sigma^{4(k-1)+2}(Q) + \sigma^{4k+2}(R) +\sigma^{4(k-1)+2}(U)+\sigma^{4(k-2)+2}(U)+\cdots+\sigma^{2}(U)
\]

\noindent is a directed cycle of length $2m$ of $J^*_{2m}$. Lastly, observe that this directed cycle is also an admissible 2-regular subdigraph of $J^*_{2m}$. This concludes the example. 
\end{example}

We will now generalize the construction given in Example \ref{ex:capsplice}. First, Definitions \ref{capdef} and \ref{continuedef} formalize the structures required for our construction. We require nine factors, all respecting the fixed external pattern $\cal X$. The structures in Definitions~\ref{capdef} and~\ref{continuedef} allow us to decompose $J_{2m}^*$ into `left', `centre' and `right' pieces  whose union is $J_{2m}^*$. The left piece will be isomorphic to the digraph obtained from $J^*_{2\ell}$ by removing the arc $x_{\ell}y_{\ell}$, while the right piece will be isomorphic to the digraph obtained from $J^*_{2\ell}$ by adding the arc $x_{0}y_{0}$. Lastly, the centre piece will be obtained from $J^*_{2\ell}$ by adding the arc $x_0y_0$ and removing the arc $x_{\ell}y_{\ell}$. See Figure~\ref{Fig:LeftMiddleRight} for an illustration of these three digraphs.

\begin{figure}[htpb!]
\begin{center}
\begin{subfigure} {0.45\textwidth}
\centering
\begin{tikzpicture}[x=1.5cm,y=1.5cm,scale=0.8]
\draw(1,0) -- (1,1);
\begin{scope}[thick, decoration={markings, mark=at position 0.33 with {\arrow{>}}}]
\draw[postaction={decorate}](2,0) -- (2,1);
\end{scope}
\foreach \x in {0,1}{
    \draw[thick] (\x,0) -- (\x+1,0);
    \draw[thick] (\x,1) -- (\x+1,1);
    \draw[thick] (\x,0) -- (\x+1,1);
    \draw[thick] (\x,1) -- (\x+1,0);
    \draw[thick] (\x,0) -- (\x+2,1);
    \draw[thick] (\x,1) -- (\x+2,0);
    \draw[thick] (\x,0) .. controls (\x+1,-0.25) .. (\x+2,0);
    \draw[thick] (\x,1) .. controls (\x+1,1.25) .. (\x+2,1);
}
\foreach \x in {0,...,3}{
    \draw[ball color=black](\x,1) circle (3pt) node[above, inner sep=10pt]{$x_{\x}$};
    \draw[ball color=black](\x,0) circle (3pt) node[below, inner sep=10pt]{$y_{\x}$};
}
\end{tikzpicture}
\caption{The left piece $\mathcal{L}$.}
\end{subfigure}
\begin{subfigure} {0.45\textwidth}
\centering
    \begin{tikzpicture} [x=1.5cm,y=1.5cm,scale=0.8]
\draw(1,0) -- (1,1);
\draw(2,0) -- (2,1);
\draw(3,0) -- (3,1);
\begin{scope}[thick, decoration={markings, mark=at position 0.5 with {\arrow{>}}}]
\draw[postaction={decorate}](0,1) -- (0,0);
\end{scope}
\foreach \x in {0,...,2}{
    \draw[thick] (\x,0) -- (\x+1,0);
    \draw[thick] (\x,1) -- (\x+1,1);
    \draw[thick] (\x,0) -- (\x+1,1);
    \draw[thick] (\x,1) -- (\x+1,0);
    \draw[thick] (\x,0) -- (\x+2,1);
    \draw[thick] (\x,1) -- (\x+2,0);
    \draw[thick] (\x,0) .. controls (\x+1,-0.25) .. (\x+2,0);
    \draw[thick] (\x,1) .. controls (\x+1,1.25) .. (\x+2,1);
}
\foreach \x in {0,...,4}{
    \draw[ball color=black](\x,1) circle (3pt) node[above, inner sep=10pt]{$x_{\x}$};
    \draw[ball color=black](\x,0) circle (3pt) node[below, inner sep=10pt]{$y_{\x}$};
}
\end{tikzpicture}
\caption{The right piece $\mathcal{R}$.}
\end{subfigure}
\\

\begin{subfigure} {1\textwidth}
\centering
\begin{tikzpicture}[x=1.5cm,y=1.5cm,scale=0.8]
\begin{scope}[shift={(4,0)}]
\foreach \x in {1,2,3}{
    \draw(\x,0) -- (\x,1);
}
\begin{scope}[thick, decoration={markings, mark=at position 0.33 with {\arrow{>}}}]
\draw[postaction={decorate}](4,0) -- (4,1);
\end{scope}
\begin{scope}[thick, decoration={markings, mark=at position 0.5 with {\arrow{>}}}]
\draw[postaction={decorate}](0,1) -- (0,0);
\end{scope}
\foreach \x in {0,...,3}{
    \draw[thick] (\x,0) -- (\x+1,0);
    \draw[thick] (\x,1) -- (\x+1,1);
    \draw[thick] (\x,0) -- (\x+1,1);
    \draw[thick] (\x,1) -- (\x+1,0);
    \draw[thick] (\x,0) -- (\x+2,1);
    \draw[thick] (\x,1) -- (\x+2,0);
    \draw[thick] (\x,0) .. controls (\x+1,-0.25) .. (\x+2,0);
    \draw[thick] (\x,1) .. controls (\x+1,1.25) .. (\x+2,1);
}
\foreach \x in {0,...,5}{
    \draw[ball color=black](\x,1) circle (3pt) node[above, inner sep=10pt]{$x_{\x}$};
    \draw[ball color=black](\x,0) circle (3pt) node[below, inner sep=10pt]{$y_{\x}$};
}
\end{scope}
\end{tikzpicture}
\caption{The centre piece $\mathcal{C}$.}   
\end{subfigure}
\end{center}
\caption{In this example, $\mathcal{L} \cup \sigma^2(\mathcal{C}) \cup \sigma^6(\mathcal{R}) = J_{18}^*$. Edges without arrows represent arcs in both directions.}\label{Fig:LeftMiddleRight}
\end{figure}
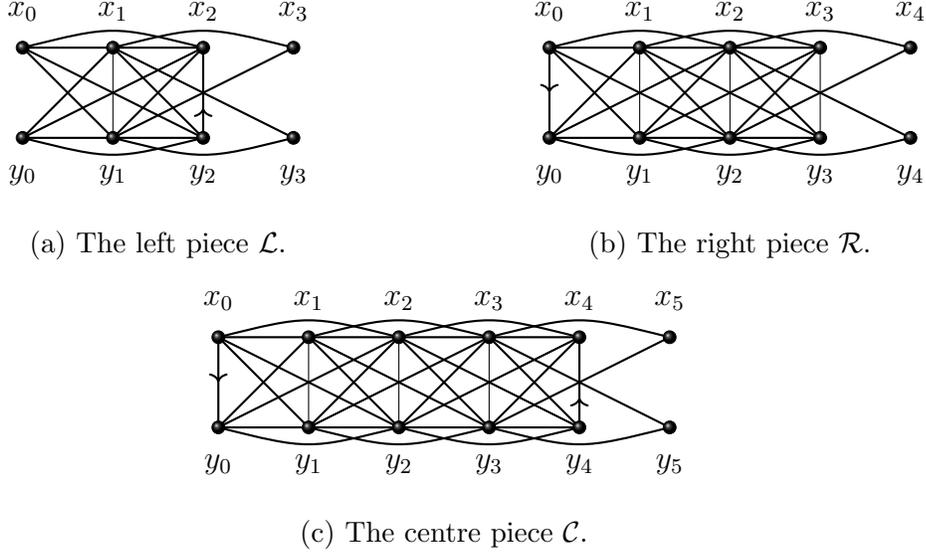

In Lemmas \ref{ComplementaryCaps} and \ref{LeftContinueRight}, we demonstrate how specific decompositions of these three digraphs into directed paths can be used to construct an admissible $F$-decomposition of $J^*_{2m}$, where $F=[d+8k, m_1, m_2, \ldots m_t]$, 
for all $d \in \{0,2,4,6\}$ and positive integers $k$. Subsequently, Lemma~\ref{GeneralFactors} constructs factors of the form $[2s]$, $s\geq 4$, which means that we will also need to construct factors with $m_i \in \{2,4,6\}$ and $t \in \{0,1,2\}$; these are also provided in Lemma~\ref{GeneralFactors}, with some small cases given in Lemma~\ref{SmallFactors}. 

First, we define the required decompositions of the left and right pieces.

\begin{definition} \label{capdef}
Let $r, \ell \geq 1$ be integers, $M=[m_1, m_2, \ldots, m_t]$ be a multiset of directed cycle lengths, and $\mathscr{E}=(E_1, E_2, \ldots, E_9)$ be an external pattern. 

\begin{enumerate}
\item 
A list of subdigraphs $\mathcal{L} = (L_1, L_2, \ldots, L_9)$ of $J_{2\ell}^*$ is an $(\mathscr{E},\ell)$-{\em left cap} if the $L_i$ are pairwise arc-disjoint, the union of their arc sets is $A(J_{2\ell}^*) \setminus \{x_{\ell} y_{\ell}\}$, and for all $i \in \{1, \ldots, 9\}$:
\begin{enumerate}

    \item $L_i$ is a directed path.
    \item  $V(L_i) \cap \{x_0, x_1, y_0, y_1\} = E_i$.
    \item For all $j \in \{2,3, \ldots, \ell-1\}$, $\{x_j,y_j\} \subseteq V(L_i)$.
     \item $s(L_i), t(L_i) \in \{x_{\ell},x_{\ell+1}, y_{\ell}, y_{\ell+1}\}$. 
\end{enumerate}

\item A list of subdigraphs $\mathcal{R} = (R_1, R_2, \ldots, R_9)$ of $J_{2r}^* \oplus \{x_0y_0\}$ is an {\em $(\mathscr{E},r,t,M)$-right cap} if the $R_i$ are pairwise arc-disjoint, the union of their arc sets is $A(J_{2r}^*) \cup \{x_0y_0\}$, and for each $i \in \{1, \ldots, 9\}$:
\begin{enumerate}
    \item $R_i$ is a vertex-disjoint union of a single directed path $P_i$ and $t$ cycles of lengths $m_1, m_2, \ldots, m_t$.
    \item For each $j \in \{0,1\}$, $x_{r+j} \in V(R_i)$ (resp.\ $y_{r+j} \in V(R_i)$) if and only if $x_j \notin E_i$ (resp.\ $y_j \notin E_i$).
    
    \item For each $j \in \{2, \ldots, r-1\}$, $\{x_j, y_j\} \subseteq V(R_i)$.
    \item $s(P_i), t(P_i) \in \{x_0,x_1,y_0,y_1\}$. 
\end{enumerate}
\noindent If $t=0$, which implies that $M=\emptyset$, we omit $t$ and $M$ from the notation and write $(\mathscr{E},r)$-right cap.
\end{enumerate}
\end{definition}

We will describe conditions which ensure that a left and right cap can be combined to create an admissible decomposition of $J_{2(\ell+r)}^*$.  The source and terminal of the directed paths in the caps clearly play an important role. Consequently, we introduce some additional terminology before proceeding.

\begin{definition}\label{InternalPatternDef}
Let $\mathscr{E}=(E_1, E_2, \ldots, E_9)$ be an external  pattern.
\begin{enumerate}
\item Let $\mathcal{L}=(L_1, \ldots, L_9)$ be an $(\mathscr{E},\ell)$-left cap.  
The {\em internal  pattern} $I_i$ of $L_i$ is the list $I_i=(s(\sigma^{-\ell}(L_i)), t(\sigma^{-\ell}(L_i)), \sigma^{-\ell}(S_i))$, where $S_i$ is the set of vertices in  $\{x_{\ell}, x_{\ell+1}, y_{\ell}, y_{\ell+1}\}$ that are internal vertices of the directed paths $L_i$.  

The {\em internal  pattern of $\mathcal{L}$} is the list $\mathcal{I}_{\mathcal{L}}=(I_1, \ldots, I_9)$.

\item Let $\mathcal{R}=(R_1, \ldots, R_9)$ be an $(\mathscr{E},r,t,M)$-right cap, and let $P_i$ be the path in $R_i$ for each $i \in \{1, \ldots, 9\}$.  
The {\em internal  pattern} $I_i'$ of $R_i$ is the list $I_i'=(t(P_i), s(P_i), \bar{S}_i)$, where $\bar{S}_i$ is the set of vertices in $\{x_0, x_{1}, y_0, y_{1}\}$ that are neither in $P_i$ nor in one of the cycles of $R_i$.  

The {\em internal  pattern of $\mathcal{R}$} is the list $\mathcal{I}_{\mathcal{R}}=(I_1', \ldots, I_9')$.
\end{enumerate}
\end{definition}

Given left and right caps, we say they have the same internal pattern if $\mathcal{I_L} = \mathcal{I_R}$. Note that, in this case, 
for each $i \in \{1, \ldots, 9\}$, $S_i = \bar{S}_i$. Furthermore, the first entry of the internal pattern of a digraph in a left cap is a source, while the first entry of the internal pattern of a digraph in a right cap is a terminal. Similarly, the second entries of the internal pattern of a left and right cap are a terminal and source, respectively. This ensures that, if the left and right cap have the same internal pattern, then $s(L_i)=t(\sigma^{\ell}(P_i))$, i.e. the source of $L_i$ is the same as the terminal of the path in $\sigma^{\ell}(R_i)$. Therefore, when the internal patterns are equal, $L+\sigma^\ell(R)$ is a directed 2-factor of $J^*_{2(\ell+r)}$.

\begin{example} 
\label{ex:capfusion}
In Figures~\ref{fig:dipathL} and~\ref{fig:dipathR}, we gave an example of a particular component of a left and right cap with the same internal patterns. The dipath $L$ from Figure~\ref{fig:dipathL} has internal pattern $(y_2, x_3, S)$, where $S=\emptyset$.As described in Example \ref{ex:capsplice}, $L+\sigma^2(R)$ is a directed cycle of length 10. This directed cycle is shown in Figure~\ref{fig:LeftRightCombined}. This concludes the example. 
\end{example}

\begin{figure}[ht]
\centering
\begin{tikzpicture}[x=1.5cm,y=1.5cm,scale=0.75]
\foreach \x in {0,...,6} {
\coordinate (x\x) at (\x,1) {};
\coordinate (y\x) at (\x,0) {};
}
\begin{scope}[thick, decoration={markings, mark=at position 0.33 with {\arrow{>}}}]
\draw[postaction={decorate}, gray] (y2) -- (x0);
\draw[postaction={decorate}, gray] (x0) -- (y1);
\draw[postaction={decorate}, gray] (y1) -- (x1);
\draw[postaction={decorate}] (y4)--(y5);
\end{scope}

\begin{scope}[thick, decoration={markings, mark=at position 0.33 with {\arrow{>}}}]
\draw[postaction={decorate}, gray] (x1) .. controls (2,0.65) .. (x3);
\draw[postaction={decorate}] (x2) .. controls (3,0.65) .. (x4);
\draw[postaction={decorate}] (y5) .. controls (4,0.45) .. (y3);
\end{scope}

\begin{scope}[thick, decoration={markings, mark=at position 0.5 with {\arrow{>}}}]
\draw[postaction={decorate}] (x4) -- (y2);
\draw[postaction={decorate}] (y3) -- (x2);
\draw[postaction={decorate}]
(x3)--(y4);
\end{scope}

\foreach \x in {3,4,5}{
	\draw[ball color=black] (y\x) circle (3pt);
	\node[below=4pt] () at (y\x) {$y_{\x}$};
}

\draw[ball color=black] (y0) circle (3pt);
\draw[ball color=white] (y1) circle (3pt);
\draw[ball color=red] (y2) circle (3pt);
\draw[ball color=white] (y6) circle (3pt);
\node[below=4pt] () at (y0) {$y_0$};
\node[below=4pt] () at (y1) {$y_{1}$};
\node[below=4pt] () at (y2) {$y_{2}$};
\node[below=4pt] () at (y6) {$y_{6}$};
    
\foreach \x in {2,4}{
	\draw[ball color=black] (x\x) circle (3pt);
	\node[above=4pt] () at (x\x) {$x_{\x}$};
}

\foreach \x in {5,6}{
	\draw[ball color=white] (x\x) circle (3pt);
	\node[above=4pt] () at (x\x) {$x_{\x}$};
}

\foreach \x in {0,1}{
	\draw[ball color=white] (x\x) circle (3pt);
	\node[above=4pt] () at (x\x) {$x_{\x}$};
}
\draw[ball color=green] (x3) circle (3pt);
\node[above=4pt] () at (x3) {$x_3$};
\end{tikzpicture}
\caption{The directed cycle $L+\sigma^2(R)$, where arcs of $L$ are drawn in grey and arcs of $\sigma^2(R)$ are drawn in black.}
\label{fig:LeftRightCombined}
\end{figure}
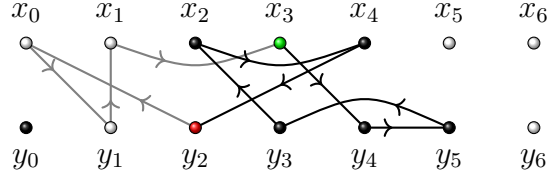

A more general example of left and right caps with matching internal patterns can be found in Appendix~\ref{Appendix:Cap_and_Continue_Piece_Picture}.  Each digraph of the right cap in the appendix consists of a directed cycle of length $4$ and a directed path which can be concatenated with the corresponding directed path in the left cap to form a directed cycle of length $10$.

\begin{definition}
\label{defn:compl}
An $(\mathscr{E},\ell)$-left cap, $\mathcal{L}$, and an $(\mathscr{E}', r, t, M)$-right cap,  $\mathcal{R}$, are {\em complementary} if  $\mathscr{E}=\mathscr{E}'$ and $\mathcal{I}_{\mathcal{L}} = \mathcal{I}_{\mathcal{R}}$. That is, they have the same external and internal patterns. 
\end{definition}

Our next step is to show how the ingredient pieces described in Definitions~\ref{capdef} and~\ref{InternalPatternDef}  can be put together to obtain an admissible decomposition of $J_{2(\ell+r)}^*$. We do so in the proof of Lemma \ref{ComplementaryCaps}, in which we show how complementary left and right caps, specifically an $(\mathscr{E},\ell)$-left cap and an  $(\mathscr{E}, r, t, M)$-right cap, give rise to an admissible decomposition of $J_{2(\ell+r)}^*$. Namely, we generalize the observations made in Example \ref{ex:capfusion} to all components of both caps by showing how the conditions given in Definitions \ref{capdef}, \ref{InternalPatternDef} and~\ref{defn:compl} ensure that we are able to construct an admissible decomposition of $J_{2(\ell+r)}^*$.

\begin{lemma} \label{ComplementaryCaps}
Let $\mathcal{L}=(L_1, \ldots, L_9)$ and $\mathcal{R}=(R_1, \ldots, R_9)$ be complementary left and right caps, respectively, such that $\mathcal{L}$ is an $(\mathscr{E},\ell)$-left cap and $\mathcal{R}$ is an $(\mathscr{E},r,t,M)$-right cap, where $M$ is a multiset $[m_1, \ldots, m_t]$ of directed cycle lengths.  For $i \in \{1, \ldots, 9\}$, let $P_i$ be the directed path in $R_i$.  If $\len(P_i)+\len(R_i)= m_0$ for all $i \in \{1, \ldots, 9\}$, then there exists an admissible $[m_0, m_1, \ldots, m_t]$-decomposition of $J^*_{2(\ell+r)}$ with external pattern $\mathscr{E}$.
\end{lemma}

\begin{proof} 
For each $i \in \{1, \ldots, 9\}$, let $D_i$ be the digraph formed by the taking the union of $L_i$ and $\sigma^{\ell}(R_i)$. We claim that the digraphs $D_1, \ldots, D_9$ form the required decomposition.   We will first verify that, for each $i \in \{1, \ldots, 9\}$, $D_i \cong [m_0, m_1, \ldots, m_t]$. To do this, it suffices to check that $L_i +\sigma^{\ell}(P_i)$ is a directed cycle of length $m_0$ that is vertex-disjoint from any directed cycle in $\sigma^{\ell}(R_i)$.

Let $(v_1, v_2, S_i)$ be the internal pattern of $L_i$ and $R_i$ where $S_i$ is the set of vertices in $\{x_0, x_1, y_0, y_1\}$ that are also internal vertices of $\sigma^{-\ell}(L_i)$. By item 2 of Definition~\ref{InternalPatternDef}, $S_i$ is also the set of vertices in $\{x_0, x_1, y_0, y_1\}$ that do not appear in $R_i$. Therefore, the internal vertices of $L_i$ are not contained in $\sigma^\ell(R_i)$. Since $\mathcal{I}_{\mathcal{L}} = \mathcal{I}_{\mathcal{R}}$, items 1 and 2 of Definition~\ref{InternalPatternDef} imply that $\sigma^{\ell}(v_1)=s(L_i)=t(\sigma^{\ell}(P_i))$ and $\sigma^{\ell}(v_2)=t(L_i)=s(\sigma^{\ell}(P_i))$. As a result, the concatenation $L_i +\sigma^{\ell}(P_i)$ is indeed possible. In addition, since the internal vertices of $L_i$ do not appear in $\sigma^\ell(P_i)$, the digraph $L_i +\sigma^{\ell}(P_i)$ is a directed cycle, which has length $m_0$ because $\len(P_i)+\len(R_i)=m_0$. Lastly, since the internal vertices of $L_i$ do not appear in digraphs of $\sigma^\ell(R_i)$, this directed cycle is also vertex-disjoint from the remaining directed cycles of $R_i$ and thus, $D_i\cong [m_0, m_1, \ldots, m_t]$, as claimed. 

We now show that $\{D_1, D_2, \ldots, D_9\}$ is a decomposition of $J_{2(\ell+r)}^*$.  It is easy to see that the union of their arcs contains all arcs of $J_{2(\ell+r)}^*$. Recall that digraphs in $\cal L$ are pairwise arc-disjoint, as are the digraphs in  $(\sigma^{\ell}(R_1), \sigma^{\ell}(R_2), \ldots, \sigma^{\ell}(R_9))$, by Definition~\ref{capdef}. Moreover, by item 2 of Definition \ref{capdef}, the union of the arc sets of $\sigma^{\ell}(R_1), \ldots, \sigma^{\ell}(R_9)$ is $A(\sigma^{\ell}(J_{2r}^*)) \cup \{x_{\ell}y_{\ell}\}$. Therefore, if $L_i$ and $\sigma^{\ell}(R_j)$ (where $i,j \in \{1, \ldots, 9\}$) are not arc-disjoint, then $L_i$ contains the arc $x_{\ell}y_{\ell}$, thereby contradicting the definition of left cap.

Lastly, we show that the decomposition $\{D_1, D_2, \ldots, D_9\}$is admissible with external  pattern $\mathscr{E}$.  Since $\mathcal{L}$ and $\mathcal{R}$ are complementary,  then for each $w \in \{x_0, x_1, y_0, y_1\}$,  items 1(b) and 2(b) of Definition \ref{capdef} jointly imply that $w$ occurs in $L_i$ if and only if $\sigma^{r}(w)$ does not occur in $R_i$.  Hence, for each $w \in \{x_0, x_1, y_0, y_1\}$, exactly one of $w$ and $\sigma^{\ell+r}(w)$ occurs in $D_i$.  Thus, the decomposition is admissible; it is easy to see that its external pattern is $\mathscr{E}$.
\end{proof}

We will construct left and right caps in order to create cycles of length $m \in \{8,10,12,14\}$ in $J_{m}^*$, as well as some other $2$-regular digraphs involving $2$- and $4$-cycles.  In order to extend our constructions to directed cycles of arbitrary even length, we now introduce the concept of a {\em centre piece}. Informally, a centre piece is a pair of directed paths that can be placed between left and right caps to create longer cycles.  These dipaths must satisfy a particular set of conditions, given in Definition \ref{continuedef}.

\begin{definition}\label{continuedef}  
Let $c>0$ be an integer. For each $i \in \{1, \ldots, 9\}$, let $C_i=(Q_i, U_i)$ denote the digraph on vertex set $\{x_i,y_i \mid i=0, \ldots, c+1\}$  comprised of two vertex-disjoint directed paths $Q_i$ and $U_i$.  The list $(C_1, \ldots, C_9)$ is called a {\em centre piece of length $c$} if:

\begin{enumerate}
\item \label{CentreArcs} The union of the arc sets of the $C_i$ is 
$\Big(A(J_{2c}^*)\cup \{x_0 y_0\}\Big) \setminus \{y_{c}x_{c}\}$, and 
\item For each $i \in \{1, \ldots, 9\}$:
\begin{enumerate}
    \item \label{CentreLength}$\len(Q_i)+\len(U_i)=2c$.  
    \item 
    \label{SourceTerminal}
    $s(Q_i)\in \{x_0, y_0, x_1, y_1\}$,  $t(Q_i)=s(\sigma^{c}(Q_i))$, $s(U_i)\in \{x_{c}, y_{c}, x_{c+1}, y_{c+1}\}$, and $t(U_i)=s(\sigma^{-c}(U_i))$.
    \item \label{InternalVertices}
    For each vertex $u \in \{x_0,x_1,y_0,y_1\} \setminus \{s(Q_i), t(U_i)\}$, exactly one of $u$ and 
    $\sigma^{c}(u)$ is in $Q_i$ or $U_i$.
    \item \label{Vertexdisjoint} Each vertex in  $\{x_2,y_2,x_3,y_3, \ldots, x_{c-1}, y_{c-1}\}$ is in precisely one of $Q_i$ or $U_i$.
\end{enumerate}
\end{enumerate}
\end{definition}

We note that Conditions \ref{SourceTerminal} and \ref{InternalVertices} of Definition~\ref{continuedef} imply that copies of a component of a centre piece can be concatenated to form a longer centre piece. That is, given a centre piece $(C_1, \ldots, C_9)$ of length $c$, where $C_i=(Q_i,U_i)$, the list $(C_1', \ldots, C_9')$, where $C_i'=(Q_i+\sigma^{c}(Q_i), \sigma^{c}(U_i) + U_i)$, is a centre piece of length $2c$. In the following, we define centre pieces of length $c=4$, which can be concatenated to obtain centre pieces of all lengths congruent to $0 \pmod 4$.  We formalize this idea in Lemma \ref{ConcatenateCentre}.

\begin{lemma} \label{ConcatenateCentre}
If there exists a centre piece of length $4$, then there exists a centre piece of length $4k$ for all $k\geqslant 1$. 
\end{lemma}

\begin{proof}
Let $\mathcal{C}=(C_1, \ldots, C_9)$ be a centre piece of length 4 such that $C_i=(Q_i,U_i)$. For each $q\in \{0,1,\ldots, k-1\}$, let $Q^q_i=\sigma^{4q}(Q_i)$ and $U^q_i=\sigma^{4q}(U_i)$, and let $C'_i=(Q_i', U_i')$, where  

\[ Q_i'=Q_i^0+Q_i^1+\cdots +Q_i^{k-1}\ \textrm{and} \ U_i'=U_i^{k}+U_i^{k-1}+\cdots +U_i^0. \]

We claim that $\mathcal{C}'=(C'_1, \ldots, C'_9)$ is a centre piece of length $4k$. To that end, we verify that all conditions of Definition \ref{continuedef} are met. 

First, we must show that $Q'_i$ and $U_i'$ are vertex-disjoint directed paths. By virtue of Condition \ref{SourceTerminal} of Definition \ref{continuedef}, we have that $t(Q_i^q)=s(Q_i^{q+1})$ for all $q\in \{0,1,\ldots, k-1\}$,  and $t(U_i^q)=s(U_i^{q-1})$ for all $q\in \{1,\ldots, k\}$. Therefore the above concatenation is in fact possible. It remains to be shown that $Q_i'$ and $U_i'$ are vertex-disjoint. Note that, if $q_1-q_2\geqslant 2$, then $Q_i^{q_1}$ and $Q_i^{q_2}$ are clearly vertex-disjoint. In addition, Conditions \ref{SourceTerminal} and \ref{Vertexdisjoint} of Definition~\ref{continuedef} jointly imply that the only common vertices of $Q_i^q$ and $Q_i^{q+1}$ are their terminal and source, respectively. By virtue of conditions \ref{SourceTerminal} and \ref{InternalVertices}, $Q'_i$ and $U'_i$ are vertex-disjoint directed paths, and by Condition \ref{CentreLength} on the constituent paths in the concatenation, the sum of their lengths is $8k$. This means that $\mathcal{C}'$ satisfies Condition \ref{CentreLength}.

Observe that, by Condition \ref{SourceTerminal}, $s(\sigma^{4}(Q_i))=t(Q_i)$. Since $Q_i^0=Q_i$, $s(\sigma^{4k}(Q_i^0))=t(Q_i^{4k})$. Therefore, $s(\sigma^{4k}(Q'_i))=t(Q_i')$ A similar reasoning applies to $U_i'$ and thus, Condition \ref{SourceTerminal} holds for $\mathcal{C}'$. Again, by a similar reasoning, it can be shown that Condition \ref{InternalVertices} also holds for $\mathcal{C}'$. 

As for Condition \ref{Vertexdisjoint}, consider $x_{r+4\ell} \in \{x_2, \ldots, x_{4k-1}\}$, where $0 \leqslant r<4$ and $0\leqslant \ell <k$ (a similar reasoning applies to $y_{4+4\ell}$). Because $Q_i'$ and $U_i'$ are vertex-disjoint, it suffices to show that $x_{r+4\ell}$ must be in at least one of $Q_i'$ or $U_i'$.  Since $\sigma^{-4\ell}(x_{r+4\ell})=x_r$ is in exactly one of $Q_i$ or $U_i$, $\cal C'$ satisfies Condition \ref{Vertexdisjoint}. 

In order to show that Condition \ref{CentreArcs} holds for $\cal C'$, it suffices to show that directed paths in $\{U'_i, U'_j, Q'_i, Q'_j\}$ are pairwise arc-disjoint when $i \neq j$. 
Clearly, $U_i'$ and $Q'_i$ are arc-disjoint; likewise for $Q'_j$ and $U_j'$. Suppose that there exists an arc $a\in A(U'_i)\cap A(U'_j)$. Without loss of generality, suppose that $a=(x_{r+4\ell}, x_{s+4\ell})$ where $0\leqslant \ell\leqslant k$ and $0\leqslant r,s <4$. 

Consider the arc $(\sigma^{-4\ell}(x_{r+4\ell}), \sigma^{-4\ell}(x_{s+4\ell}))=(x_r, x_s)$. By definition, this arc lies in $A(U_i)$ and $A(U_j)$, thereby contradicting the fact that these dipaths are arc-disjoint. A similar reasoning applies to the remaining pairs of dipaths, namely $Q_i'$ and $Q_j'$, $U_i'$ and $Q_j'$, and $U_j'$ and $Q_i'$. \end{proof}

Next, we give precise conditions that a centre piece of length $4k$ must satisfy in order to form an admissible decomposition of $J_{2(\ell+4k+r)}^*$ when paired with a complementary pair of left and right caps.

\begin{definition}
Given a centre piece $\mathcal{C}=(C_1,\ldots, C_9)$, the {\em internal  pattern} of $C_i=(Q_i,U_i)$ is the list $I''_i=(t(U_i),s(Q_i),\bar{S})$, where $\bar{S}$ is the set of vertices in $\{x_0,x_1,y_0,y_1\}$ that are not in $V(Q_i)\cup V(U_i)$.  The {\em internal  pattern of $\mathcal{C}$} is the list $\mathcal{I}_{\mathcal{C}} = (I_1'', \ldots, I_9'')$. 

We say that a left (resp.\ right) cap $\mathcal{L}$ (resp.\ $\mathcal{R}$) and a centre piece $\mathcal{C}$ are {\em complementary} if they have the same internal patterns, that is, $\mathcal{I}_{\mathcal{L}} = \mathcal{I}_{\mathcal{C}}$ (resp.\ $\mathcal{I}_{\mathcal{R}} = \mathcal{I}_{\mathcal{C}}$).
\end{definition}

Although we use the term {\em complementary} between left and right caps as well as between a cap and a centre piece, this term describes a similar concept: that the caps or the cap and centre piece can be joined to create directed cycles of longer length.  We note that, if $\mathcal{L}$ and $\mathcal{R}$ are complementary left and right caps, and a centre piece $\mathcal{C}$ is complementary to either, then it is complementary to both. 

\begin{lemma}\label{LeftContinueRight}
Let $\mathcal{L}=(L_1,\ldots,L_9)$ and $\mathcal{R}=(R_1, \ldots, R_9)$ be complementary left and right caps satisfying the hypothesis of Lemma~\ref{ComplementaryCaps} and let $\mathcal{C}$ be a centre piece of length $4$.  If $\mathcal{L}$ and $\mathcal{C}$ are complementary,  then there exists an admissible $[m_0+8k, m_1, m_2, \ldots, m_t]$-decomposition of $J_{2(\ell+4k+r)}^*$ with external pattern $\mathscr{E}$ for each integer $k \geq 0$.
\end{lemma}

\begin{proof}
The case $k=0$ is Lemma~\ref{ComplementaryCaps} and thus, we assume $k \geq 1$.  By Lemma~\ref{ConcatenateCentre}, there is a centre piece $\mathcal{C}_{4k}=(C_1, \ldots, C_9)$ of length $4k$ with the same internal pattern as $\mathcal{C}$. 

For each $i \in \{1, \ldots, 9\}$, we have $C_i=(Q_i, U_i)$.  Furthermore, recall that $R_i$ is comprised of a single directed path $P_i$ and a set of $t$ directed cycles of lengths $m_1, m_2, \ldots, m_t$. Note that the set of directed cycles associated with $R_i$ is possibly empty and recall that the sum of the lengths of $L_i$ and $P_i$ is $m_0$. 

\noindent Consider the subdigraph \[D_i = L_i + \sigma^{\ell}(Q_i) + \sigma^{\ell+4k}(P_i) + \sigma^{\ell}(U_i)\] of $J_{2(\ell+4k+r)}^*$.  We will show that $D_i$ is a directed cycle of length $m_0+8k$. 

We first argue that the concatenations are well-defined. We note that, since $\mathcal{I}_{\mathcal{L}} = \mathcal{I}_{\mathcal{C}} = \mathcal{I}_{\mathcal{C}_{4k}}$, then $t(L_i)=s(\sigma^{\ell}(Q_i))$ and $s(L_i)=t(\sigma^{\ell}(U_i))$. Furthermore, since $\mathcal{I}_{\mathcal{R}}=\mathcal{I}_{\mathcal{C}}=\mathcal{I}_{\mathcal{C}_{4k}}$, then $s(\sigma^{\ell+4k}(P_i))=t(\sigma^{\ell}(Q_i))$ and $t(\sigma^{\ell+4k}(P_i))=s(\sigma^{\ell}(U_i))$. Therefore, the concatenations are well-defined.

Next, we argue that, other than its endpoints, the digraph $D_i$ has no repeated vertices. By Definition~\ref{continuedef}, $\sigma^{\ell}(Q_i)$ and $\sigma^{\ell}(U_i)$ are vertex-disjoint.  Because $\mathcal{I}_{\mathcal{C}_{4k}}=\mathcal{I}_{\mathcal{L}}=\mathcal{I}_{\mathcal{R}}$, the dipaths $L_i$ and $\sigma^{\ell+4k}(P_i)$ have no vertices in common with $\sigma^{\ell}(Q_i)$ or  $\sigma^{\ell}(U_i)$ other than their endpoints. Therefore, we have that $D_i$ is a directed cycle. Moreover, note that the directed cycle $D_i$ has length $m_0+8k$ since $\len(\sigma^{\ell}(Q_i))+\len(\sigma^{\ell}(U_i))=8k$ by Condition~(\ref{CentreLength}) of Definition~\ref{continuedef} and $\len(L_i)+\len(P_i)=m_0$ by hypothesis.

Note that $D_i$ is vertex-disjoint from the cycles in $R_i$ (if any) because these cycles are vertex-disjoint from $P_i$. We then construct the digraph $F_i=L_i \oplus \sigma^{\ell}(C_i) \oplus \sigma^{\ell+4k}(R_i)$, which is the union of the cycle $D_i$ of length $m_0+8k$ with the directed cycles of lengths $m_1, m_2, \ldots, m_t$ of $R_i$. Consequently, $F_i$ is a 2-regular subdigraph of $J_{2(\ell+4k+r)}^*$. That $F_i$ is admissible follows by a similar reasoning to the proof of Lemma~\ref{ComplementaryCaps}.

It remains to be shown that $\mathcal{F}=(F_1,\ldots, F_9)$ is a decomposition of $J_{2(\ell+4k+r)}^*$ with external pattern $\mathscr{E}$.  To verify that $\mathcal{F}$ is a decomposition, it suffices to show that the digraphs $F_1, \ldots, F_9$ are pairwise arc-disjoint.  Consider $F_i$ and $F_j$, where $i \neq j$.  To see that they are arc-disjoint, first note that $L_i$ and $L_j$ are arc-disjoint by definition, as are $\sigma^{\ell}(C_i)$ and $\sigma^{\ell}(C_j)$, and $\sigma^{\ell+4k}(R_i)$ and $\sigma^{\ell+4k}(R_j)$.  Further, it is clear that since $k \geq 1$, $L_i$ and $\sigma^{\ell+4k}(R_j)$ are arc-disjoint, as are $L_j$ and $\sigma^{\ell+4k}(R_i)$.  Since the arc set of $L_i$ is contained in $A(J_{2\ell}^*) \setminus \{x_{\ell}y_{\ell}\}$ and the arc set of $\sigma^{\ell}(C_j)$ is contained in $(A(\sigma^{\ell}(J_{8k}^*)) \cup \{x_{\ell}y_{\ell}\}) \setminus\{y_{4k+\ell}x_{4k+\ell}\}$, it is straightforward to check that $L_i$ and $\sigma^{\ell}(C_j)$ are arc-disjoint.  All remaining pairs of digraphs in $\{L_i, \sigma^{\ell}(C_i), \sigma^{\ell}(C_j), \sigma^{\ell+4k}(R_i), \sigma^{\ell+4k}(R_j)\}$ can be shown to be arc-disjoint by a similar reasoning. Clearly, the decomposition $\mathcal{F}$ has the same external pattern as $\mathcal{L}$ and $\mathcal{R}$, which is $\mathscr{E}$.  \end{proof}

As a consequence of Lemma~\ref{LeftContinueRight}, in order to construct an admissible decomposition of  $J^*_{2m}$, where $m=\ell+4k+r$, it suffices to construct a pair of complementary left and right caps satisfying the hypothesis of Lemma~\ref{ComplementaryCaps} and a centre piece of length four which is complementary to these caps. We now do this for some particular 2-regular digraphs. 

Recall that each of our caps will have external pattern

\begin{multline*}
\mathcal{X}=(
\{y_1\}, \{x_0,x_1,y_1\}, \{x_1,y_0,y_1\}, \{x_0,x_1,y_0,y_1\}, \{y_1\}, \{x_0,x_1,y_0\}, \{x_1,y_1\}, \\ \{x_1\}, \{x_0,x_1,y_0,y_1\}
).
\end{multline*}

\begin{lemma} \label{GeneralFactors}
Let $F$ be a $2$-regular digraph of order $2m$.  There exists an admissible $F$-decomposition of $J_{2m}^*$ with external pattern $\mathcal{X}$ in each of the following cases:
     \begin{multicols}{2}
    \begin{enumerate}
        \item $F=[2s]$, $s \geq 4$;
        \item $F=[2s,2]$, $s \geq 4$;
        \item $F=[2s,2^2]$, $s \geq 4$;
        \item $F=[2s,4]$, $s \geq 5$.
    \end{enumerate}
    \end{multicols}
\end{lemma}

\begin{proof}
We give left and right caps and a centre piece of length $4$ that satisfy the conditions of Lemma~\ref{LeftContinueRight}.  In each case, we use the same $(\mathcal{X},2)$-left cap$\mathcal{L}=(L_1, \ldots, L_9)$ and centre piece $\mathcal{C}=(C_1, \ldots, C_9)$.  The construction for $F=[2s,4]$, including the left cap $\mathcal{L}$ and centre piece $\mathcal{C}$, as well as the right cap for this case, is illustrated in Appendix~\ref{Appendix:Cap_and_Continue_Piece_Picture}.
    
\noindent Each digraph in $\mathcal{L}$ consists of a single path, given below:
\[
\begin{array}{rlrl}
L_1: & \langle y_2,y_1,x_2\rangle; & L_6: & \langle y_2,y_0,x_2,x_0,x_1,y_3\rangle; \\
L_2: & \langle y_2,x_0,y_1,x_1,x_3\rangle; & L_7: & \langle y_3,x_1,y_1,x_3 \rangle; \\
L_3: & \langle x_3,y_1,y_0,x_1,y_2\rangle; & L_8: & \langle y_2,x_1,x_2\rangle; \\
L_4: & \langle x_3,x_1,x_0,y_2,x_2,y_0,y_1,y_3\rangle ; &L_9: & \langle y_3,y_1,x_0,x_2,x_1,y_0,y_2\rangle. \\
L_5: & \langle x_2,y_1,y_2\rangle; \\
\end{array}
\]

Each digraph $C_i=(P_i,Q_i)$ in $\mathcal{C}$ consists of two vertex-disjoint paths $P_i$, which has source in $\{x_0,x_1,y_0,y_1\}$ and terminal in $\{x_4,x_5,y_4,y_5\}$, and $Q_i$, which has source in $\{x_4,x_5,y_4,y_5\}$ and terminal in $\{x_0,x_1,y_0,y_1\}$. These are listed as follows:

\[
\begin{array}{rll rll}
C_1: & (\langle x_0,y_2,x_2,x_1,y_3,x_4\rangle, \langle y_4,x_3,y_1,y_0\rangle); & C_6: & (\langle y_1,x_1,x_3,y_3,y_5\rangle, \langle y_4,y_2,x_4, x_2,y_0\rangle);  \\
C_2: & (\langle x_1,x_0,x_2,x_3,x_5\rangle, \langle y_4,y_3,y_1,y_2,y_0\rangle); & C_7: & (\langle x_1,x_2,y_3,x_5\rangle, \langle y_5,x_3,y_2,x_0,y_0,y_1\rangle);  \\
C_3: & (\langle y_0,y_2,y_1,y_3,y_4\rangle, \langle 
x_5,x_3,x_2,x_0,x_1\rangle); & C_8: & (\langle x_0,y_1,x_3,x_4\rangle, \langle y_4,x_2,y_2,y_3,x_1,y_0 \rangle); \\
C_4: & (\langle y_1,x_2,y_4,x_4,x_3,y_5\rangle, \langle x_5,y_3,y_2,x_1\rangle); & C_9: & (\langle y_0,x_2,x_4, y_2, y_4\rangle, \langle y_5,y_3,x_3,x_1,y_1\rangle).\\
C_5: & (\langle y_0,x_1,y_2,x_3,y_4\rangle, \langle x_4,y_3,x_2,y_1,x_0\rangle); \\
\end{array}
\]

For each $F \in \{[2s],[2s,2],[2s,2^2],[2s,4]$, we give a right cap for the smallest value of $s$, say $s_0$, in each congruence class modulo $4$ so that the sum of the path lengths in $L_i$ and $R_i$ sum to $2s_0$; the result for larger $s$ then follows by Lemma~\ref{LeftContinueRight}.  

We first consider $F=[2s]$.  An $(\mathcal{X},s_0-2)$-right cap for each case is as follows:
\begin{itemize}
\item $s_0=4$:
\[
\begin{array}{rlrl}
R_1: & \langle x_0,y_1,x_3,x_1,x_2,y_2,y_0 \rangle; & R_6: & \langle y_1, y_3, x_1, y_0 \rangle; \\
R_2: & \langle x_1,x_0,y_2,y_1,y_0 \rangle; &R_7: & \langle x_1, y_2, x_0, y_0, x_2, y_1 \rangle; \\
R_3: & \langle y_0, y_1, x_0, x_2, x_1 \rangle; & R_8: & \langle x_0, x_1, y_3, y_1, y_2, x_2, y_0 \rangle; \\
R_4: & \langle y_1, x_1 \rangle; & R_9: & \langle y_0, x_1, y_1 \rangle. \\
R_5: & \langle y_0, y_2, x_1, x_3, y_1, x_2, x_0 \rangle; \\
\end{array}
\]
\item $s_0=5$: 
\[
\begin{array}{rl rl}
R_1: & \langle x_0, x_1, y_1, x_3, y_3, x_2, x_4, y_2, y_0 \rangle; & R_6: & \langle y_1, x_2, y_4, y_2, x_1, y_0 \rangle; \\
R_2: & \langle x_1, y_2, y_3, y_1, x_0, x_2, y_0 \rangle; & R_7: & \langle x_1, y_3, y_2, x_3, x_2, x_0, y_0, y_1 \rangle; \\
R_3: & \langle y_0, x_2, x_3, y_2, x_0, y_1, x_1 \rangle; & R_8: & \langle x_0, y_2, y_4, x_2, y_3, x_1, x_3, y_1, y_0 \rangle; \\
R_4: & \langle y_1, y_2, x_2, x_1 \rangle; & R_9: & \langle y_0, x_1, x_2, y_2, y_1 \rangle.  \\
R_5: & \langle y_0, y_2, x_4, x_2, y_1, y_3, x_3, x_1, x_0 \rangle; \\
\end{array}
\]
\item $s_0=6:$
\[
\begin{array}{rlrl}
R_1: & \langle x_0, x_1, y_1, x_2, x_3, x_5, y_3, x_4, y_4, y_2, y_0 \rangle; & R_6: & \langle y_1, x_3, y_5, y_3, x_2, y_2, x_1, y_0 \rangle; \\
R_2: & \langle x_1, x_0, y_1, y_2, y_3, x_3, y_4, x_2, y_0 \rangle; & R_7: & \langle x_1, y_3, y_4, x_3, x_4, y_2, x_2, x_0, y_0, y_1 \rangle; \\
R_3: & \langle y_0, y_2, x_0, x_2, x_4, x_3, y_3, y_1, x_1 \rangle; & R_8: & \langle x_0, y_2, y_4, x_4, y_3, y_5, x_3, x_1, x_2, y_1, y_0 \rangle; \\
R_4: & \langle y_1, y_3, y_2, x_3, x_2, x_1 \rangle; & R_9: & \langle y_0, x_2, y_3, x_1, x_3, y_2, y_1 \rangle. \\
R_5: & \langle y_0, x_1, y_2, x_4, x_2, y_4, y_3, x_5, x_3, y_1, x_0 \rangle; \\
\end{array}
\]
\item $s_0=7$: 
\[
\begin{array}{rl}
R_1: &  \langle x_{0}, x_{1}, y_{1}, x_{2}, x_{3}, y_{3}, x_{5}, y_{5}, x_{4}, x_{6}, y_{4}, y_{2}, y_{0}\rangle; \\
R_2: & \langle x_{1}, x_{0}, y_{1}, y_{2}, x_{3}, x_{4}, y_{3}, y_{5}, y_{4}, x_{2}, y_{0}\rangle; \\
R_3: & \langle y_{0}, y_{1}, x_{0}, x_{2}, y_{2}, y_{3}, x_{4}, x_{5}, y_{4}, x_{3}, x_{1}\rangle; \\
R_4: & \langle y_{1}, x_{3}, y_{4}, x_{4}, y_{2}, x_{2}, y_{3}, x_{1}\rangle; \\
R_5: & \langle y_{0}, x_{2}, y_{4}, x_{6}, x_{4}, y_{5}, x_{5}, y_{3}, x_{3}, y_{1}, x_{1}, y_{2}, x_{0}\rangle; \\
R_6: & \langle y_{1}, y_{3}, y_{2}, y_{4}, y_{6}, x_{4}, x_{3}, x_{2}, x_{1}, y_{0}\rangle;  \\
R_7: & \langle x_{1}, y_{3}, y_{4}, y_{5}, x_{3}, x_{5}, x_{4}, x_{2}, x_{0}, y_{0}, y_{2}, y_{1}\rangle;  \\
R_8: & \langle x_{0}, y_{2}, x_{1}, x_{2}, x_{4}, y_{6}, y_{4}, x_{5}, x_{3}, y_{5}, y_{3}, y_{1}, y_{0}\rangle;  \\
R_9: & \langle y_{0}, x_{1}, x_{3}, y_{2}, x_{4}, y_{4}, y_{3}, x_{2}, y_{1}\rangle. \\
\end{array}
\]
\end{itemize}
See Appendix \ref{A:continue} for the remaining cases. \end{proof}

\begin{lemma} \label{SmallFactors}
Let $m \in \{3, 4, 5, 6\}$. There exists an admissible $F$-decomposition of $J_{2m}^*$ with external pattern $\mathcal{X}$ for each of the following $2$-regular digraphs $F$ of order $2m$:
    \[
    [2,4], [2,6], [2^2,4], [2^2,6], [2^2,4^2], [2^3], [2^3,4], [4^2], [4^3], [4,6], [4,8], [6].
    \]
\end{lemma}

\begin{proof}
The required decompositions may be found in Appendix~\ref{Small 12-decomps}.  A $[4,8]$-decomposition of $J_{12}^*$ is also illustrated in Appendix~\ref{[4,8]}.
\end{proof}

Now that we have established the existence of admissible $F$-decompositions of $J_{2m}^*$ for cases where $F$ has one component or a small number of components with small cycle lengths, we may employ these results with Lemma~\ref{SpliceDecomps} to produce decompositions for more general $2$-regular graphs $F$.  This approach will then allow us to construct directed $2$-factorizations of $W_{2m}^*$.

\begin{theorem}\label{J-decomposition}
Let $m \geq 4$ and $F \neq [2^m]$ be a bipartite $2$-regular of order $2m$. There exists an admissible $F$-decomposition of $J_{2m}^*$. 
\end{theorem}

\begin{proof}
Write $F=[m_1^{\alpha_1}, m_2^{\alpha_2}, \ldots, m_t^{\alpha_t}]$ where $2 \leq m_1 \leq m_2 \leq \cdots \leq m_t$ and $\alpha_1, \ldots, \alpha_t$ are positive integers. We consider three cases depending on the value of $m_1$.
\vskip 1ex

\noindent CASE 1: $m_1 \geq 6$. For each $i \in \{1, \ldots, t\}$, there exists an admissible $[m_i]$-decomposition of $J_{2m_i}^*$ with external  pattern $\mathcal{X}$, by Lemma~\ref{SmallFactors} if $m_i=6$ and by Lemma~\ref{GeneralFactors} otherwise.  Thus, by Lemma~\ref{SpliceDecomps}, there exists an admissible $F$-decomposition of $J_{2m}^*$ with external  pattern $\mathcal{X}$.
\vskip 1ex

\noindent CASE 2: $m_1=4$. We consider cases depending on the value of $\alpha_1$.  
\vskip 1ex

\noindent\underline{Sub-case 2.1:} $\alpha_1 \geq 2$. Write $\alpha_1 = 2\beta+3\gamma$, where $\beta$ and $\gamma$ are non-negative integers.  By Lemma~\ref{SmallFactors}, there exist an admissible $[4^2]$-decomposition of $J_8^*$ and an admissible $[4^3]$-decompositions of $J_{12}^*$, each with external  pattern $\mathcal{X}$.  The existence of an admissible $[4^{\alpha_1}]$-decomposition of $J_{4\alpha_1}^*$ follows by Lemma~\ref{SpliceDecomps}.  Using Case 1, there exists an admissible $[m_2^{\alpha_2}, \ldots, m_t^{\alpha_t}]$-decomposition of $J_{2m-4\alpha_1}^*$.  These two latter decompositions are compatible, each having external  pattern $\mathcal{X}$, and hence, by Lemma~\ref{SpliceDecomps}, there is an admissible $[m_1^{\alpha_1}, \ldots, m_t^{\alpha_t}]$-decomposition of $J_{2m}^*$ with external  pattern $\mathcal{X}$. 
\vskip 1ex

\noindent\underline{Sub-case 2.2:}  $\alpha_1=1$. Note that there is an admissible $[4,m_2]$-decomposition of $J_{4+m_2}^*$ by Lemma~\ref{GeneralFactors} in the case $m_2 > 6$ and Lemma~\ref{SmallFactors} if $m_2=6$.  By Case 1, there is an admissible $[m_2^{\alpha_2-1}, m_3^{\alpha_3}, \ldots, m_t^{\alpha_t}]$-decomposition of $J_{2m-(4+m_2)}^*$.  As both of these decompositions have external  pattern $\mathcal{X}$, Lemma~\ref{SpliceDecomps} implies the existence of an admissible $[m_1^{\alpha_1}, \ldots, m_t^{\alpha_t}]$-decomposition of $J_{2m}^*$ with external  pattern $\mathcal{X}$.
\vskip 1ex

\noindent CASE 3: $m_1=2$. Since $\mathcal{F} \neq [2^m]$, we have that $t \geq 2$.  We proceed depending on the congruence class of $\alpha_1$ modulo $3$.
\vskip 1ex

\noindent \underline{Sub-case 3.1:} $\alpha_1 \equiv 0 \pmod{3}$. The result follows by Lemma~\ref{SpliceDecomps}, using the $[2^3]$-decomposition from Lemma~\ref{SmallFactors} and an admissible $[m_2^{\alpha_2}, \ldots, m_t^{\alpha_t}]$-decomposition of $J_{2m-2\alpha_1}^*$ from Case 1 or 2.
\vskip 1ex

\noindent\underline{Sub-case 3.2:} $\alpha_1 \equiv 1 \pmod{3}$. We again use Lemma~\ref{SpliceDecomps}, this time with an admissible $[2,m_2]$-decomposition of $J_{2+m_2}^*$ from Lemma~\ref{GeneralFactors} (if $m_2 \geq 8$) or Lemma~\ref{SmallFactors} (if $m_2 \in \{4,6\}$), as well as an admissible $[2^3]$-decomposition from Lemma~\ref{SmallFactors} (if $\alpha_1>1$) and an $[m_2^{\alpha_2-1}, m_3^{\alpha_3}, \cdots m_t^{\alpha_t}]$-decomposition of $J_{2m-2\alpha_1-m_2}^*$ from Case 1 or 2.
\vskip 1ex

\noindent \underline{Sub-case 3.3:} $\alpha_1 \equiv 2 \pmod{3}$. This case is similar to Subcase 3.2, except that we use a $[2^2,m_2]$-decomposition of $J_{4+m_2}^*$ from Lemma~\ref{GeneralFactors} or~\ref{SmallFactors} in place of a $[2,m_2]$-decomposition. \end{proof}

\begin{corollary}\label{W-decomposition}
Let $m \geq 4$ be an integer and  $F \neq [2^m]$ is a bipartite directed $2$-factor of order $2m$. There exists an $F$-factorization of $W_{2m}^*$. 
\end{corollary}

\begin{proof}
    The result follows directly from Theorem~\ref{J-decomposition} by applying Lemma~\ref{JtoW}.
\end{proof}

We now prove the main result of this paper. 
{
\renewcommand{\thetheorem}{\ref{thm:main}}
\begin{theorem}
Let $F$ be a bipartite directed $2$-factor of order $n \equiv 2 \pmod{4}$.  There exists a solution to $\op^*(F)$ if and only if $F \notin \{[4],[6]\}$.
\end{theorem}
\addtocounter{theorem}{-1}
}

\begin{proof}
    If $F$ is uniform, then the result follows by Theorem~\ref{Uniform}; in particular, this means that we may assume that $F \neq [2^{n/2}]$.  Note that~\cite[Theorem 9.1]{KadriSajna} solves the problem for orders $n < 14$.  For $n \geq 14$, the result follows from Theorem~\ref{W-reduction} and Corollary~\ref{W-decomposition}.  
\end{proof}

Theorem~\ref{thm:main} completely resolves the case of the directed Oberwolfach problem in which all directed cycles are of even length and the number of vertices is congruent to $2$ modulo $4$.  
The natural next step is to consider the case where the number of vertices is congruent to $0$ modulo $4$. We have embarked on this study; although similar methods may apply, this case appears to be more complicated.

\section{Acknowledgements}
A.C.\ Burgess acknowledges support from the Natural Sciences and Engineering Research Council of Canada,  Discovery Grant RGPIN-2025-04633. 

\noindent P.\ H.\ Danziger acknowledges support from the Natural Sciences and Engineering Research Council of Canada,  Discovery Grant RGPIN-2022-0381. 

\noindent A.\ Lacaze-Masmonteil acknowledges the support of the Pacific Institute for the Mathematical Sciences and Centre National de la Recherche Scientifique Postdoctoral Fellowship. 

\section{Conflict of Interest Statement}

The authors declare no potential conflicts of interest.

\pagebreak

\appendix

\section{An admissible $[4,8]$-decomposition of $J_{12}^*$}
\label{[4,8]}

The following set $\mathcal{F}=\{F_1, \ldots, F_9\}$ of nine digraphs forms an admissible $[4,8]$-decomposition of $J_{12}^*$ with external pattern \[
\mathcal{X}=(
\{y_1\}, \{x_0,x_1,y_1\}, \{x_1,y_0,y_1\}, \{x_0,x_1,y_0,y_1\}, \{y_1\}, \{x_0,x_1,y_0\}, \{x_1,y_1\}, \{x_1\},\{x_0,x_1,y_0,y_1\}).
\]
\medskip

\noindent
$F_1$: External pattern $\{y_1\}$.
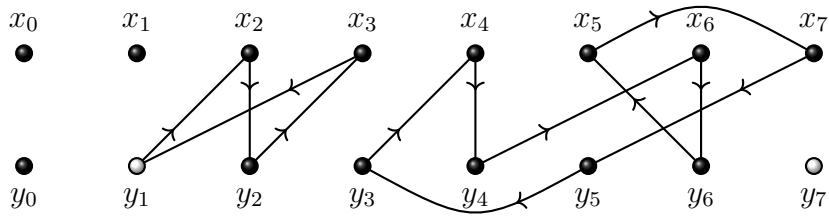
\begin{center}
\begin{tikzpicture}[x=1.5cm,y=1.5cm,scale=1]
\foreach \x in {0,...,7} {
	\coordinate (x\x) at (\x,1) {};
	\coordinate (y\x) at (\x,0) {};
}

\begin{scope}[thick, decoration={markings, mark=at position 0.6 with {\arrow{>}}}]
\draw[postaction={decorate}] (y6) -- (x5);
\end{scope}

\begin{scope}[thick, decoration={markings, mark=at position 0.33 with {\arrow{>}}}]
\draw[postaction={decorate}] (y1) -- (x2);
\draw[postaction={decorate}] (x2) -- (y2);
\draw[postaction={decorate}] (y2) -- (x3);
\draw[postaction={decorate}] (x3) -- (y1);
\draw[postaction={decorate}] (y3) -- (x4);
\draw[postaction={decorate}] (x4) -- (y4);
\draw[postaction={decorate}] (y4) -- (x6);
\draw[postaction={decorate}] (x6) -- (y6);
\draw[postaction={decorate}] (x5) .. controls (6, 1.55) .. (x7);
\draw[postaction={decorate}] (x7) -- (y5);
\draw[postaction={decorate}] (y5) .. controls (4, -0.55) .. (y3);
\foreach \x in {0,...,7} {
	\draw[ball color=black] (x\x) circle (3pt);
	\node[above=4pt] () at (x\x) {$x_{\x}$};
	\node[below=4pt] () at (y\x) {$y_{\x}$};
}
\foreach \x in {0,2,3,4,5,6} {
	\draw[ball color=black] (y\x) circle (3pt);
}
\draw[ball color=white] (y1) circle (3pt);
\draw[ball color=white] (y7) circle (3pt);
\end{scope}
\end{tikzpicture}
\end{center}
\medskip

\noindent
$F_2$: External pattern $\{x_0,x_1,y_1\}$.
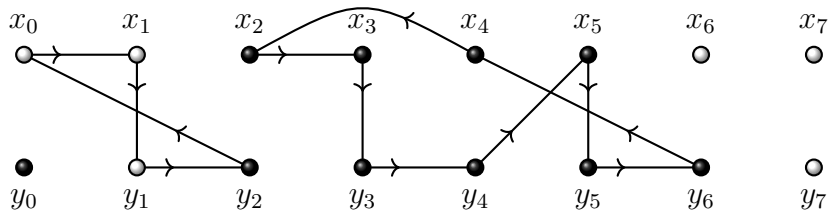
\begin{center}
\begin{tikzpicture}[x=1.5cm,y=1.5cm,scale=1]
\foreach \x in {0,...,7} {
	\coordinate (x\x) at (\x,1) {};
	\coordinate (y\x) at (\x,0) {};
}
\begin{scope}[thick, decoration={markings, mark=at position 0.33 with {\arrow{>}}}]
\draw[postaction={decorate}] (x0) -- (x1);
\draw[postaction={decorate}] (x1) -- (y1);
\draw[postaction={decorate}] (y1) -- (y2);
\draw[postaction={decorate}] (y2) -- (x0);
\draw[postaction={decorate}] (x2) -- (x3);
\draw[postaction={decorate}] (x3) -- (y3);
\draw[postaction={decorate}] (y3) -- (y4);
\draw[postaction={decorate}] (y4) -- (x5);
\draw[postaction={decorate}] (x5) -- (y5);
\draw[postaction={decorate}] (y5) -- (y6);
\draw[postaction={decorate}] (y6) -- (x4);
\draw[postaction={decorate}] (x4) .. controls (3, 1.55) .. (x2);
\foreach \x in {0,...,7} {
	\node[above=4pt] () at (x\x) {$x_{\x}$};
	\node[below=4pt] () at (y\x) {$y_{\x}$};
}
\foreach \x in {2,3,4,5} {
	\draw[ball color=black] (x\x) circle (3pt);
}
\foreach \x in {0,1,6,7} {
	\draw[ball color=white] (x\x) circle (3pt);
}
\foreach \x in {0,2,3,4,5,6}{
	\draw[ball color=black] (y\x) circle (3pt);
}
\draw[ball color=white] (y7) circle (3pt);
\draw[ball color=white] (y1) circle (3pt);
\end{scope}
\end{tikzpicture}
\end{center}
\medskip

\noindent
$F_3$: External pattern $\{x_1,y_0,y_1\}$.
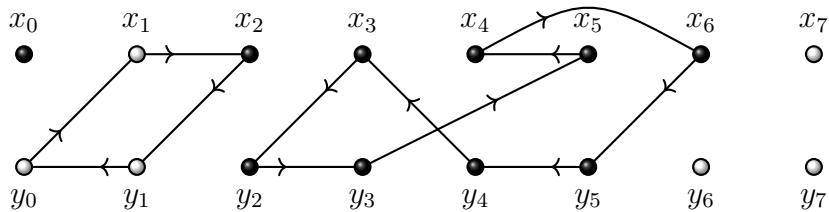
\begin{center}
\begin{tikzpicture}[x=1.5cm,y=1.5cm,scale=1]
\foreach \x in {0,...,7} {
	\coordinate (x\x) at (\x,1) {};
	\coordinate (y\x) at (\x,0) {};
}

\begin{scope}[thick, decoration={markings, mark=at position 0.6 with {\arrow{>}}}]
\draw[postaction={decorate}] (y3) -- (x5);
\draw[postaction={decorate}] (y4) -- (x3);
\end{scope}

\begin{scope}[thick, decoration={markings, mark=at position 0.33 with {\arrow{>}}}]
\draw[postaction={decorate}] (y0) -- (x1);
\draw[postaction={decorate}] (x1) -- (x2);
\draw[postaction={decorate}] (x2) -- (y1);
\draw[postaction={decorate}] (y1) -- (y0);
\draw[postaction={decorate}] (y2) -- (y3);
\draw[postaction={decorate}] (x5) -- (x4);
\draw[postaction={decorate}] (x4) .. controls (5, 1.55) .. (x6);
\draw[postaction={decorate}] (x6) -- (y5);
\draw[postaction={decorate}] (y5) -- (y4);
\draw[postaction={decorate}] (x3) -- (y2);
\foreach \x in {0,...,7} {
	\node[above=4pt] () at (x\x) {$x_{\x}$};
	\node[below=4pt] () at (y\x) {$y_{\x}$};
}
\foreach \x in {0,2,3,...,6}{
	\draw[ball color=black] (x\x) circle (3pt);
}
\draw[ball color=white] (x1) circle (3pt);
\draw[ball color=white] (x7) circle (3pt);
\foreach \x in {2,...,5}{
	\draw[ball color=black] (y\x) circle (3pt);
}
\foreach \x in {0,1,6,7}{
	\draw[ball color=white] (y\x) circle (3pt);
}
\end{scope}
\end{tikzpicture}
\end{center}
\medskip

\noindent
$F_4$: External pattern $\{x_0,x_1,y_0,y_1\}$.
\begin{center}
\begin{tikzpicture}[x=1.5cm,y=1.5cm,scale=1]
\foreach \x in {0,...,7} {
	\coordinate (x\x) at (\x,1) {};
	\coordinate (y\x) at (\x,0) {};
}

\begin{scope}[thick, decoration={markings, mark=at position 0.6 with {\arrow{>}}}]
\draw[postaction={decorate}] (x2) -- (y0);
\draw[postaction={decorate}] (y1) -- (x0);
\end{scope}

\begin{scope}[thick, decoration={markings, mark=at position 0.67 with {\arrow{>}}}]
\draw[postaction={decorate}] (x1) .. controls (2, 1.55) .. (x3);
\end{scope}

\begin{scope}[thick, decoration={markings, mark=at position 0.33 with {\arrow{>}}}]
\draw[postaction={decorate}] (x0) .. controls (1, 1.55) .. (x2);
\draw[postaction={decorate}] (y0) -- (y1);
\draw[postaction={decorate}] (x3) -- (y5);
\draw[postaction={decorate}] (y5) -- (x4);
\draw[postaction={decorate}] (x4) -- (x5);
\draw[postaction={decorate}] (x5) -- (y4);
\draw[postaction={decorate}] (y4) -- (y3);
\draw[postaction={decorate}] (y3) -- (y2);
\draw[postaction={decorate}] (y2) -- (x1);
\foreach \x in {0,...,7} {
	\node[above=4pt] () at (x\x) {$x_{\x}$};
	\node[below=4pt] () at (y\x) {$y_{\x}$};
}
\foreach \x in {0,1,6,7}{
	\draw[ball color=white] (x\x) circle (3pt);
	\draw[ball color=white] (y\x) circle (3pt);
}
\foreach \x in {2,...,5}{
	\draw[ball color=black] (x\x) circle (3pt);
	\draw[ball color=black] (y\x) circle (3pt);
}
\end{scope}
\end{tikzpicture}
\end{center}
\medskip

\pagebreak
\noindent
$F_5$: External pattern $\{y_1\}$.
\begin{center}
\begin{tikzpicture}[x=1.5cm,y=1.5cm,scale=1]
\foreach \x in {0,...,7} {
	\coordinate (x\x) at (\x,1) {};
	\coordinate (y\x) at (\x,0) {};
}

\begin{scope}[thick, decoration={markings, mark=at position 0.6 with {\arrow{>}}}]
\draw[postaction={decorate}] (x2) -- (y3);
\end{scope}

\begin{scope}[thick, decoration={markings, mark=at position 0.67 with {\arrow{>}}}]
\draw[postaction={decorate}] (y3) .. controls (2, -0.55) .. (y1);
\draw[postaction={decorate}] (y2) .. controls (3, -0.55) .. (y4);
\draw[postaction={decorate}] (x6) .. controls (5, 1.55) .. (x4);
\end{scope}

\begin{scope}[thick, decoration={markings, mark=at position 0.33 with {\arrow{>}}}]
\draw[postaction={decorate}] (y1) -- (x3);
\draw[postaction={decorate}] (x3) -- (x2);
\draw[postaction={decorate}] (y4) .. controls (5, -0.55) .. (y6);
\draw[postaction={decorate}] (y6) -- (y5);
\draw[postaction={decorate}] (y5) -- (x7);
\draw[postaction={decorate}] (x7) .. controls (6, 1.55) .. (x5);
\draw[postaction={decorate}] (x5) -- (x6);
\draw[postaction={decorate}] (x4) -- (y2);
\foreach \x in {0,...,7} {
	\draw[ball color=black] (x\x) circle (3pt);
	\node[above=4pt] () at (x\x) {$x_{\x}$};
	\node[below=4pt] () at (y\x) {$y_{\x}$};
}
\foreach \x in {0,2,3,...,6}{
	\draw[ball color=black] (y\x) circle (3pt);
}
\draw[ball color=white](y1) circle (3pt);
\draw[ball color=white](y7) circle (3pt);
\end{scope}
\end{tikzpicture}
\end{center}
\medskip

\noindent 
$F_6$: External pattern $\{x_0,x_1,y_0\}$.
\begin{center}
\begin{tikzpicture}[x=1.5cm,y=1.5cm,scale=1]
\foreach \x in {0,...,7} {
	\coordinate (x\x) at (\x,1) {};
	\coordinate (y\x) at (\x,0) {};
}
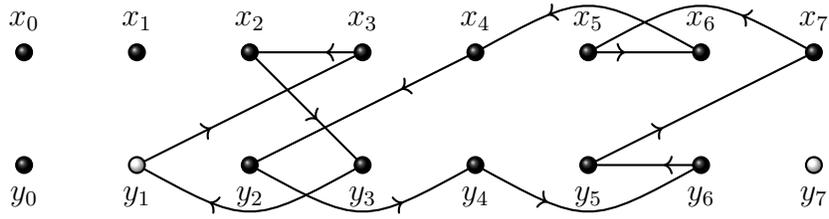
\begin{scope}[thick, decoration={markings, mark=at position 0.33 with {\arrow{>}}}]
\draw[postaction={decorate}] (y3) .. controls (4, -0.55) .. (y5);
\draw[postaction={decorate}] (y5) .. controls (6, -0.55) .. (y7);
\draw[postaction={decorate}] (y7) -- (x5);
\draw[postaction={decorate}] (x5) -- (y3);
\draw[postaction={decorate}] (x0) -- (y2);
\draw[postaction={decorate}] (y2) .. controls (1, -0.55) .. (y0);
\draw[postaction={decorate}] (y0) -- (x2);
\draw[postaction={decorate}] (x2) -- (y4);
\draw[postaction={decorate}] (y4) -- (x4);
\draw[postaction={decorate}] (x4) -- (x3);
\draw[postaction={decorate}] (x3) .. controls (2, 1.55) .. (x1);
\draw[postaction={decorate}] (x1) -- (x0);
\foreach \x in {0,...,7} {
	\node[above=4pt] () at (x\x) {$x_{\x}$};
	\node[below=4pt] () at (y\x) {$y_{\x}$};
}
\foreach \x in {2,...,5}{
	\draw[ball color=black] (x\x) circle (3pt);
}
\foreach \x in {0,1,6,7}{
	\draw[ball color=white] (x\x) circle (3pt);
}
\foreach \x in {1,...,5,7}{
	\draw[ball color=black] (y\x) circle (3pt);
}
\draw[ball color=white] (y0)  circle (3pt);
\draw[ball color=white] (y6)  circle (3pt);
\end{scope}
\end{tikzpicture}
\end{center}
\medskip

\noindent
$F_7$: External pattern $\{x_1,y_1\}$.
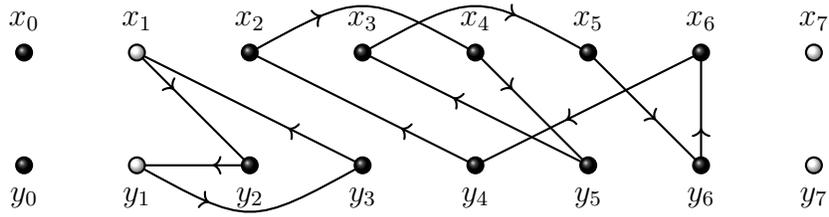
\begin{center}
\begin{tikzpicture}[x=1.5cm,y=1.5cm,scale=1]
\foreach \x in {0,...,7} {
	\coordinate (x\x) at (\x,1) {};
	\coordinate (y\x) at (\x,0) {};
}

\begin{scope}[thick, decoration={markings, mark=at position 0.6 with {\arrow{>}}}]
\draw[postaction={decorate}] (x6) -- (y4);
\draw[postaction={decorate}] (x5) -- (y6);
\draw[postaction={decorate}] (y5) -- (x3);
\end{scope}

\begin{scope}[thick, decoration={markings, mark=at position 0.67 with {\arrow{>}}}]
\draw[postaction={decorate}] (x3) .. controls (4, 1.55) .. (x5);
\end{scope}

\begin{scope}[thick, decoration={markings, mark=at position 0.33 with {\arrow{>}}}]
\draw[postaction={decorate}] (x1) -- (y2);
\draw[postaction={decorate}] (y2) -- (y1);
\draw[postaction={decorate}] (y1) .. controls (2, -0.55) .. (y3);
\draw[postaction={decorate}] (y3) -- (x1);
\draw[postaction={decorate}] (x2) .. controls (3, 1.55) .. (x4);
\draw[postaction={decorate}] (x4) -- (y5);
\draw[postaction={decorate}] (y6) -- (x6);
\draw[postaction={decorate}] (y4) -- (x2);
\foreach \x in {0,...,7} {
	\node[above=4pt] () at (x\x) {$x_{\x}$};
	\node[below=4pt] () at (y\x) {$y_{\x}$};
}
\foreach \x in {0,2,3,...,6}{
	\draw[ball color=black] (x\x) circle (3pt);
	\draw[ball color=black] (y\x) circle (3pt);
}
\foreach \x in {1,7}{
    \draw[ball color=white] (x\x) circle (3pt);
    \draw[ball color=white] (y\x) circle (3pt);
}
\end{scope}
\end{tikzpicture}
\end{center}
\medskip

\noindent
$F_8$: External pattern $\{x_1\}$
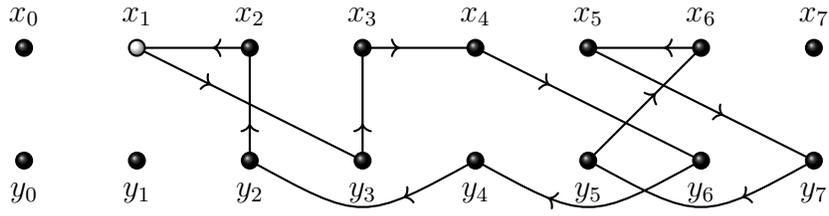
\begin{center}
\begin{tikzpicture}[x=1.5cm,y=1.5cm,scale=1]
\foreach \x in {0,...,7} {
	\coordinate (x\x) at (\x,1) {};
	\coordinate (y\x) at (\x,0) {};
}
\begin{scope}[thick, decoration={markings, mark=at position 0.6 with {\arrow{>}}}]
\draw[postaction={decorate}] (x5) -- (y7);
\draw[postaction={decorate}] (y5) -- (x6);
\end{scope}

\begin{scope}[thick, decoration={markings, mark=at position 0.67 with {\arrow{>}}}]
\draw[postaction={decorate}] (y6) .. controls (5, -0.55) .. (y4);
\end{scope}

\begin{scope}[thick, decoration={markings, mark=at position 0.33 with {\arrow{>}}}]
\draw[postaction={decorate}] (y7) .. controls (6, -0.55) .. (y5);
\draw[postaction={decorate}] (x6) -- (x5);
\draw[postaction={decorate}] (x1) -- (y3);
\draw[postaction={decorate}] (y3) -- (x3);
\draw[postaction={decorate}] (x3) -- (x4);
\draw[postaction={decorate}] (x4) -- (y6);
\draw[postaction={decorate}] (y4) .. controls (3, -0.55) .. (y2);
\draw[postaction={decorate}] (y2) -- (x2);
\draw[postaction={decorate}] (x2) -- (x1);
\foreach \x in {0,...,7} {
	\node[above=4pt] () at (x\x) {$x_{\x}$};
	\node[below=4pt] () at (y\x) {$y_{\x}$};
}
\foreach \x in {0,2,3,...,7}{
	\draw[ball color=black] (y\x) circle (3pt);
    \draw[ball color=black] (x\x) circle (3pt);
}
\draw[ball color=white] (x1) circle (3pt);
\draw[ball color=black] (y1) circle (3pt);
\end{scope}
\end{tikzpicture}
\end{center}
\medskip

\noindent
$F_9$: External pattern $\{x_0,x_1,y_0,y_1\}$.
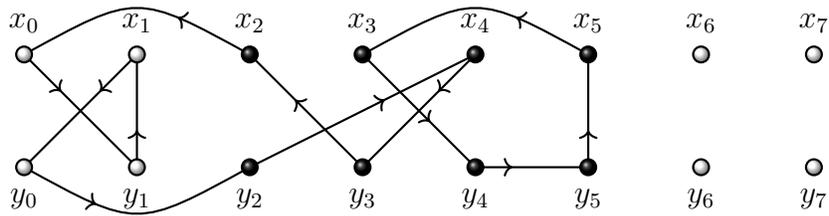
\begin{center}
\begin{tikzpicture}[x=1.5cm,y=1.5cm,scale=1]
\foreach \x in {0,...,7} {
	\coordinate (x\x) at (\x,1) {};
	\coordinate (y\x) at (\x,0) {};
}

\begin{scope}[thick, decoration={markings, mark=at position 0.6 with {\arrow{>}}}]
\draw[postaction={decorate}] (y2) -- (x4);
\draw[postaction={decorate}] (y3) -- (x2);
\draw[postaction={decorate}] (x3) -- (y4);
\end{scope}

\begin{scope}[thick, decoration={markings, mark=at position 0.33 with {\arrow{>}}}]
\draw[postaction={decorate}] (y4) -- (y5);
\draw[postaction={decorate}] (y5) -- (x5);
\draw[postaction={decorate}] (x5) .. controls (4, 1.55) .. (x3);
\draw[postaction={decorate}] (x0) -- (y1);
\draw[postaction={decorate}] (y1) -- (x1);
\draw[postaction={decorate}] (x1) -- (y0);
\draw[postaction={decorate}] (y0) .. controls (1, -0.55) .. (y2);
\draw[postaction={decorate}] (x4) -- (y3);
\draw[postaction={decorate}] (x2) .. controls (1, 1.55) .. (x0);
\foreach \x in {0,...,7} {
	\node[above=4pt] () at (x\x) {$x_{\x}$};
	\node[below=4pt] () at (y\x) {$y_{\x}$};
}
\foreach \x in {0,1,6,7}{
	\draw[ball color=white] (x\x) circle (3pt);
	\draw[ball color=white] (y\x) circle (3pt);
}
\foreach \x in {2,...,5}{
	\draw[ball color=black] (x\x) circle (3pt);
	\draw[ball color=black] (y\x) circle (3pt);
}
\end{scope}
\end{tikzpicture}
\end{center}
\medskip

\section{Caps and centre piece for a $[4,2t]$-decomposition of $J_{2t+4}^*$, $t \equiv 1 \pmod{4}$} \label{Appendix:Cap_and_Continue_Piece_Picture}

Consider the left cap $(L_1, \ldots, L_9)$, centre piece $(C_1, \ldots, C_9)$ and right cap $(R_1, \ldots, R_9)$ used to construct an admissible $[2t,4]$-decomposition of $J_{2t+4}^*$ with external pattern $\mathcal{X}$, where $t \equiv 1 \pmod{4}$, $t \geq 5$.  For each $i \in \{1, \ldots, 9\}$, $L_i \oplus \sigma^2(C_i) \oplus \sigma^6(C_i) \oplus \cdots \oplus \sigma^{4s+2}(C_i) \oplus \sigma^{4s+6}(R_i)$ consists of one directed cycle of length $10+8s$ and one of length $4$.  In particular, $L_i \oplus \sigma^2(R_i)$ is an admissible $2$-regular graph with one directed cycle of length $10$ and one of length $4$, while  
$L_i \oplus \sigma^2(C_i) \oplus \sigma^6(R_i)$ consists of one directed cycle of length $18$ and one of length $4$.

In the figures below, for each $i \in \{1, \ldots, 9\}$, we specifically illustrate $L_i$, $\sigma^2(C_i)$ and $\sigma^6(R_i)$; these three digraphs combine to form a $[4,18]$-decomposition of $J_{22}^*$.  For each value of $i$, we state the external pattern of $L_i$ and $R_i$, and the internal pattern of $L_i$, $C_i$ and $R_i$.
\bigskip

\noindent
$L_1$, $C_1$, $R_1$: External pattern $\{y_1\}$, internal pattern $(y_0,x_0,\emptyset)$.  
\begin{center}

\end{center}

\section{Right caps for Lemma~\ref{GeneralFactors}}
\label{A:continue}
In this appendix, we give the remaining right caps required for the proof of Lemma \ref{GeneralFactors}.

For $F=[2s,2]$, we give an $(\mathcal{X}, s_0-1, 1, \{2\})$-right cap in each case. This means that each $R_i$ is the disjoint union of a directed cycle of length 2 and a directed path.
\begin{itemize}
\item $s_0=4:$
\[
\begin{array}{rlrl}
R_1: & \langle x_{0}, y_{1}, x_{3}, y_{3}, x_{1}, x_{2}, y_{0} \rangle, (y_{2}, x_{4}); & R_6: & \langle y_{1}, x_{2}, x_{1}, y_{0} \rangle, (y_{2}, y_{4}); \\
R_2: & \langle x_{1}, y_{1}, x_{0}, y_{2}, y_{0} \rangle, (x_{2}, y_{3}); & R_7: & \langle x_{1}, x_{0}, y_{0}, y_{2}, y_{3}, y_{1} \rangle, (x_{2}, x_{3}); \\
R_3: & \langle y_{0}, y_{1}, y_{2}, x_{3}, x_{1} \rangle, (x_{0}, x_{2}); & R_8: & \langle x_{0}, x_{1}, y_{3}, x_{3}, y_{2}, y_{1}, y_{0} \rangle, (x_{2}, y_{4}); \\
R_4: & \langle y_{1}, x_{1} \rangle, (x_{2}, y_{2}); & R_9: & \langle y_{0}, x_{2}, y_{1} \rangle, (x_{1}, y_{2}). \\
R_5: & \langle y_{0}, x_{1}, x_{3}, y_{1}, y_{3}, y_{2}, x_{0} \rangle, (x_{2}, x_{4}); \\
\end{array}
\]
\item $s_0=5:$
\[
\begin{array}{rl rl}
R_1: & \langle x_{0}, x_{1}, y_{1}, x_{2}, x_{3}, x_{4}, y_{4}, y_{2}, y_{0} \rangle, (y_{3}, x_{5}); &R_6: & \langle y_{1}, x_{3}, x_{2}, y_{2}, x_{1}, y_{0} \rangle, (y_{3}, y_{5}); \\
R_2: & \langle x_{1}, x_{0}, y_{1}, y_{3}, y_{2}, x_{2}, y_{0} \rangle, (x_{3}, y_{4}); & R_7: & \langle x_{1}, x_{3}, y_{2}, x_{4}, x_{2}, x_{0}, y_{0}, y_{1} \rangle, (y_{3}, y_{4}); \\
R_3: & \langle y_{0}, x_{2}, x_{4}, x_{3}, y_{3}, y_{1}, x_{1} \rangle, (x_{0}, y_{2}); & R_8: & \langle x_{0}, x_{2}, y_{4}, x_{4}, y_{3}, x_{1}, y_{2}, y_{1}, y_{0} \rangle, (x_{3}, y_{5}); \\
R_4: & \langle y_{1}, y_{2}, x_{3}, x_{1} \rangle, (x_{2}, y_{3}); & R_9: & \langle y_{0}, y_{2}, y_{3}, x_{3}, y_{1} \rangle, (x_{1}, x_{2}). \\
R_5: & \langle y_{0}, x_{1}, y_{3}, x_{4}, y_{2}, y_{4}, x_{2}, y_{1}, x_{0} \rangle, (x_{3}, x_{5}); \\
\end{array}
\]
\item $s_0=6:$
\[
\begin{array}{rl}
R_1: & \langle x_{0}, x_{1}, y_{1}, x_{2}, x_{3}, y_{3}, x_{4}, x_{6}, y_{4}, y_{2}, y_{0} \rangle, (x_{5}, y_{5}); \\
R_2: & \langle x_{1}, x_{0}, y_{1}, y_{2}, x_{3}, x_{4}, y_{3}, x_{2}, y_{0} \rangle, (y_{4}, y_{5}); \\
R_3: & \langle y_{0}, y_{1}, x_{0}, y_{2}, x_{2}, y_{3}, x_{5}, x_{3}, x_{1} \rangle, (x_{4}, y_{4}); \\
R_4: & \langle y_{1}, x_{3}, x_{2}, y_{4}, y_{3}, x_{1} \rangle, (y_{2}, x_{4}); \\
R_5: & \langle y_{0}, y_{2}, y_{4}, x_{6}, x_{4}, x_{5}, y_{3}, y_{1}, x_{1}, x_{2}, x_{0} \rangle, (x_{3}, y_{5}); \\
R_6: & \langle y_{1}, y_{3}, x_{3}, y_{4}, x_{2}, y_{2}, x_{1}, y_{0} \rangle, (x_{4}, y_{6}); \\
R_7: & \langle x_{1}, y_{3}, y_{5}, x_{4}, x_{3}, y_{2}, x_{0}, y_{0}, x_{2}, y_{1} \rangle, (y_{4}, x_{5}); \\
R_8: & \langle x_{0}, x_{2}, x_{1}, x_{3}, x_{5}, x_{4}, y_{5}, y_{3}, y_{2}, y_{1}, y_{0} \rangle, (y_{4}, y_{6}); \\
R_9: & \langle y_{0}, x_{1}, y_{2}, y_{3}, y_{4}, x_{3}, y_{1} \rangle, (x_{2}, x_{4}). \\
\end{array}
\]
\item $s_0=7:$
\[
\begin{array}{rl}
R_1: & \langle x_{0}, x_{1}, y_{1}, x_{2}, x_{3}, y_{3}, x_{4}, x_{5}, x_{6}, y_{6}, y_{4}, y_{2}, y_{0} \rangle, (y_{5}, x_{7}); \\
R_2: & \langle x_{1}, x_{0}, y_{1}, y_{2}, x_{3}, x_{4}, y_{3}, x_{5}, y_{4}, x_{2}, y_{0} \rangle, (y_{5}, y_{6}); \\
R_3: & \langle y_{0}, y_{1}, x_{0}, x_{2}, y_{2}, y_{3}, y_{4}, x_{5}, x_{4}, x_{3}, x_{1} \rangle, (y_{5}, x_{6}); \\
R_4: & \langle y_{1}, x_{3}, y_{2}, y_{4}, x_{4}, x_{2}, y_{3}, x_{1} \rangle, (x_{5}, y_{5}); \\
R_5: & \langle y_{0}, y_{2}, y_{1}, x_{1}, x_{3}, y_{5}, x_{4}, y_{6}, x_{6}, y_{4}, y_{3}, x_{2}, x_{0} \rangle, (x_{5}, x_{7}); \\
R_6: & \langle y_{1}, y_{3}, y_{2}, x_{4}, y_{5}, y_{4}, x_{3}, x_{2}, x_{1}, y_{0} \rangle, (x_{5}, y_{7}); \\
R_7: & \langle x_{1}, y_{3}, y_{5}, x_{3}, y_{4}, x_{6}, x_{4}, y_{2}, x_{0}, y_{0}, x_{2}, y_{1} \rangle, (x_{5}, y_{6}); \\
R_8: & \langle x_{0}, y_{2}, x_{1}, x_{2}, y_{4}, y_{6}, x_{4}, x_{6}, x_{5}, y_{3}, x_{3}, y_{1}, y_{0} \rangle, (y_{5}, y_{7}); \\
R_9: & \langle y_{0}, x_{1}, y_{2}, x_{2}, x_{4}, y_{4}, y_{5}, y_{3}, y_{1} \rangle, (x_{3}, x_{5}). \\
\end{array}
\]
\end{itemize}

\pagebreak
For $F=[2s, 2^2]$, we give an $(\mathcal{X}, s_0, 1, \{2,2\})$-right cap in each case.
\begin{itemize}
\item $s_0=4:$
\[
\begin{array}{rl}
R_1: & \langle x_{0}, x_{1}, y_{1}, x_{2}, x_{4}, y_{2}, y_{0}\rangle, (x_{3}, y_{4}), (y_{3}, x_{5}); \\
R_2: & \langle x_{1}, x_{0}, y_{2}, y_{1}, y_{0}\rangle, (x_{2}, x_{3}), (y_{3}, y_{4}); \\
R_3: &  \langle y_{0}, y_{2}, x_{4}, y_{3}, x_{1}\rangle, (x_{0}, x_{2}), (y_{1}, x_{3}); \\
R_4: & \langle y_{1}, x_{1}\rangle, (x_{2}, y_{3}), (y_{2}, x_{3}); \\
R_5: & \langle y_{0}, x_{1}, x_{2}, y_{2}, y_{3}, y_{1}, x_{0}\rangle, (x_{3}, x_{5}), (x_{4}, y_{4}); \\
R_6: & \langle y_{1}, y_{2}, x_{2}, y_{0}\rangle, (x_{1}, x_{3}), (y_{3}, y_{5};) \\
R_7: & \langle x_{1}, y_{3}, y_{2}, x_{0}, y_{0}, y_{1}\rangle, (x_{2}, y_{4}), (x_{3}, x_{4}); \\
R_8: & \langle x_{0}, y_{1}, y_{3}, x_{4}, x_{2}, x_{1}, y_{0}\rangle, (y_{2}, y_{4}), (x_{3}, y_{5}); \\
R_9: & \langle y_{0}, x_{2}, y_{1}\rangle, (x_{1}, y_{2}), (x_{3}, y_{3}). \\
\end{array}
\]
\item $s_0=5:$

\[
\begin{array}{rl}
R_1: & \langle x_{0}, x_{2}, y_{3}, y_{5}, x_{4}, x_{6}, y_{4}, y_{2}, y_{0}\rangle,  (x_{1}, y_{1}), (x_{3}, x_{5}); \\
R_2: &  \langle x_{1}, y_{3}, x_{4}, y_{5}, y_{4}, x_{2}, y_{0}\rangle, (x_{0}, y_{1}), (y_{2}, x_{3}); \\
R_3: &  \langle y_{0}, y_{1}, x_{2}, x_{0}, y_{2}, y_{3}, x_{1}\rangle, (x_{3}, x_{4}), (y_{4}, x_{5});  \\
R_4: & \langle y_{1}, y_{3}, x_{2}, x_{1}\rangle, (y_{2}, x_{4}), (x_{3}, y_{4}); \\
R_5: & \langle y_{0}, x_{2}, y_{4}, x_{6}, x_{4}, y_{3}, y_{1}, y_{2}, x_{0}\rangle, (x_{1}, x_{3}), (x_{5}, y_{5}); \\
R_6: & \langle y_{1}, x_{3}, x_{2}, y_{2}, x_{1}, y_{0}\rangle, (y_{3}, y_{4}), (x_{4}, y_{6}); \\
R_7: & \langle x_{1}, x_{0}, y_{0}, y_{2}, y_{4}, y_{5}, x_{3}, y_{1}\rangle, (x_{2}, x_{4}), (y_{3}, x_{5}); \\
R_8: & \langle x_{0}, x_{1}, x_{2}, x_{3}, y_{5}, y_{3}, y_{2}, y_{1}, y_{0}\rangle, (x_{4}, x_{5}), (y_{4}, y_{6}); \\
R_9: & \langle y_{0}, x_{1}, y_{2}, x_{2}, y_{1}\rangle, (x_{3}, y_{3}), (x_{4}, y_{4}). \\
\end{array}
\]

\item $s_0=6:$

\[
\begin{array}{rl}
R_1: & \langle x_{0}, x_{2}, x_{4}, x_{5}, x_{7}, y_{5}, x_{6}, y_{6}, y_{4}, y_{2}, y_{0}\rangle, (x_{1}, y_{1}), (x_{3}, y_{3}); \\
R_2: &\langle x_{1}, y_{3}, x_{4}, y_{5}, y_{6}, x_{5}, y_{4}, x_{2}, y_{0}\rangle, (x_0, y_1), (y_2,x_3); \\
R_3: & \langle y_{0}, x_{2}, y_{4}, x_{5}, x_{4}, x_{6}, y_{5}, x_{3}, x_{1}\rangle, (x_0,y_2), (y_1, y_3); \\
R_4: & \langle y_{1}, y_{2}, x_{4}, y_{3}, x_{2}, x_{1}\rangle, (x_{3}, x_{5}), (y_{4}, y_{5}); \\
R_5: & \langle y_{0}, y_{1}, x_{3}, x_{4}, y_{6}, y_{5}, x_{7}, x_{5}, y_{3}, x_{1}, x_{0}\rangle, (x_{2}, y_{2}), (y_{4}, x_{6}); \\
R_6: & \langle y_{1}, x_{2}, x_{3}, y_{5}, x_{4}, y_{2}, x_{1}, y_{0}\rangle,  (y_{3}, y_{4}), (x_{5}, y_{7}); \\
R_7: &  \langle x_{1}, x_{2}, x_{0}, y_{0}, y_{2}, y_{4}, y_{6}, x_{4}, x_{3}, y_{1}\rangle, (y_{3}, y_{5}), (x_{5}, x_{6}); \\
R_8: & \langle x_{0}, x_{1}, y_{2}, y_{3}, x_{5}, y_{6}, x_{6}, x_{4}, x_{2}, y_{1}, y_{0}\rangle, (x_{3}, y_{4}), (y_{5}, y_{7});\\
R_9: & \langle y_{0}, x_{1}, x_{3}, x_{2}, y_{3}, y_{2}, y_{1}\rangle, (x_{4}, y_{4}), (x_{5}, y_{5}) .\\
\end{array}
\]

\item $s_0=7:$

\[
\begin{array}{rl}
R_1: &  \langle x_{0}, x_{2}, x_{4}, x_{5}, y_{5}, x_{7}, y_{7}, x_{6}, x_{8}, y_{6}, y_{4}, y_{2}, y_{0}\rangle, (x_{1}, y_{1}), (x_{3}, y_{3}); \\
R_2: &  \langle x_{1}, y_{3}, x_{4}, y_{5}, x_{5}, y_{7}, y_{6}, x_{6}, y_{4}, x_{2}, y_{0}\rangle, (x_{0}, y_{1}), (y_{2}, x_{3}); \\
R_3: & \langle y_{0}, x_{2}, y_{4}, x_{4}, y_{6}, x_{5}, x_{7}, x_{6}, y_{5}, y_{3}, x_{1}\rangle, (x_{0}, y_{2}), (y_{1}, x_{3}); \\
R_4: & \langle y_{1}, y_{2}, y_{4}, x_{6}, x_{4}, y_{3}, x_{2}, x_{1}\rangle,  (x_{3}, x_{5}), (y_{5}, y_{6}); \\
R_5: & \langle y_{0}, y_{1}, y_{3}, y_{4}, y_{6}, x_{8}, x_{6}, x_{7}, x_{5}, x_{4}, x_{3}, x_{2}, x_{0}\rangle, (x_{1}, y_{2}), (y_{5}, y_{7}); \\
R_6: & \langle y_{1}, x_{2}, y_{2}, x_{4}, y_{4}, y_{3}, y_{5}, x_{3}, x_{1}, y_{0}\rangle, (x_{5}, x_{6}), (y_{6}, y_{8}); \\
R_7: &\langle x_{1}, x_{0}, y_{0}, y_{2}, x_{2}, x_{3}, x_{4}, x_{6}, y_{7}, x_{5}, y_{3}, y_{1}\rangle, (y_{4}, y_{5}), (y_{6}, x_{7}); \\
R_8: & \langle x_{0}, x_{1}, x_{2}, y_{3}, x_{5}, y_{6}, y_{7}, x_{7}, y_{5}, x_{4}, y_{2}, y_{1}, y_{0}\rangle, (x_{3}, y_{4}), (x_{6}, y_{8}); \\
R_9: & \langle y_{0}, x_{1}, x_{3}, y_{5}, x_{6}, y_{6}, x_{4}, x_{2}, y_{1}\rangle, (y_{2}, y_{3}), (y_{4}, x_{5}). \\
\end{array}
\]
\end{itemize}
Finally, for $F=[2s, 4]$, we give an $(\mathcal{X}, s_0, 1, \{4\})$-right cap in each case. Note that in this case, $R_i$ is comprised of the disjoint union of a directed cycle of length 4 and a directed path. 
\begin{itemize}
\item $s_0=5:$
\[
\begin{array}{rl}
R_1: & \langle x_{0}, x_{1}, y_{1}, x_{2}, x_{3}, x_{5}, y_{3}, y_{2}, y_{0}\rangle, (x_{4}, y_{5}, y_{4}, x_{6}); \\
R_2: & \langle x_{1}, x_{0}, y_{1}, y_{2}, y_{4}, x_{2}, y_{0}\rangle, (x_{3}, y_{3}, y_{5}, x_{4}); \\
R_3: & \langle y_{0}, y_{2}, x_{0}, x_{2}, y_{3}, y_{1}, x_{1}\rangle, (x_{3}, x_{4}, x_{5}, y_{4}); \\
R_4: & \langle y_{1}, y_{3}, x_{3}, x_{1}\rangle, (x_{2}, y_{4}, x_{4}, y_{2}); \\
R_5: & \langle y_{0}, x_{2}, x_{1}, y_{3}, x_{5}, y_{5}, x_{3}, y_{1}, x_{0}\rangle, (y_{2}, x_{4}, x_{6}, y_{4}); \\
R_6: & \langle y_{1}, x_{3}, x_{2}, y_{2}, x_{1}, y_{0}\rangle, (y_{3}, y_{4}, y_{6}, x_{4}); \\
R_7: & \langle x_{1}, y_{2}, y_{3}, x_{4}, x_{2}, x_{0}, y_{0}, y_{1}\rangle, (x_{3}, y_{4}, y_{5}, x_{5}); \\
R_8: & \langle x_{0}, y_{2}, x_{3}, y_{5}, y_{3}, x_{1}, x_{2}, y_{1}, y_{0}\rangle, (x_{4}, y_{6}, y_{4}, x_{5}); \\
R_9: & \langle y_{0}, x_{1}, x_{3}, y_{2}, y_{1}\rangle, (x_{2}, x_{4}, y_{4}, y_{3}). \\
\end{array}
\]

\item $s_0=6:$
\[
\begin{array}{rl}
R_1: & \langle x_{0}, x_{1}, y_{1}, x_{2}, x_{3}, y_{3}, x_{4}, x_{6}, y_{4}, y_{2}, y_{0}\rangle, (x_{5}, y_{6}, y_{5}, x_{7}); \\
R_2: & \langle x_{1}, x_{0}, y_{1}, y_{2}, x_{3}, x_{4}, y_{3}, x_{2}, y_{0}\rangle, (y_{4}, x_{5}, y_{5}, y_{6}); \\
R_3: & \langle y_{0}, y_{1}, x_{0}, x_{2}, y_{2}, x_{4}, y_{4}, y_{3}, x_{1}\rangle, (x_{3}, y_{5}, x_{6}, x_{5}); \\
R_4: & \langle y_{1}, x_{3}, x_{5}, x_{4}, y_{2}, x_{1}\rangle, (x_{2}, y_{3}, y_{5}, y_{4}); \\
R_5: & \langle y_{0}, y_{2}, y_{1}, x_{1}, x_{3}, y_{4}, x_{6}, y_{6}, x_{4}, x_{2}, x_{0}\rangle, (y_{3}, x_{5}, x_{7}, y_{5}); \\
R_6: & \langle y_{1}, y_{3}, y_{4}, x_{3}, y_{2}, x_{2}, x_{1}, y_{0}\rangle, (x_{4}, x_{5}, y_{7}, y_{5}); \\
R_7: & \langle x_{1}, y_{3}, y_{2}, x_{0}, y_{0}, x_{2}, y_{4}, y_{5}, x_{3}, y_{1}\rangle, (x_{4}, y_{6}, x_{5}, x_{6}); \\
R_8: & \langle x_{0}, y_{2}, y_{4}, y_{6}, x_{6}, y_{5}, y_{7}, x_{5}, y_{3}, y_{1}, y_{0}\rangle, (x_{1}, x_{2}, x_{4}, x_{3}); \\
R_9: & \langle y_{0}, x_{1}, y_{2}, y_{3}, x_{3}, x_{2}, y_{1}\rangle, (x_{4}, y_{5}, x_{5}, y_{4}). \\
\end{array}
\]
\item $s_0=7:$
\[
\begin{array}{rl}
R_1: &  \langle x_{0}, x_{1}, y_{1}, x_{2}, x_{3}, y_{3}, x_{4}, x_{5}, x_{7}, y_{5}, y_{4}, y_{2}, y_{0}\rangle, (x_{6}, y_{7}, y_{6}, x_{8}); \\
R_2: & \langle x_{1}, x_{0}, y_{1}, y_{2}, x_{3}, x_{4}, y_{3}, x_{5}, y_{4}, x_{2}, y_{0}\rangle, (y_{5}, x_{6}, y_{6}, y_{7}); \\
R_3: & \langle y_{0}, y_{1}, x_{0}, x_{2}, y_{3}, y_{4}, x_{5}, y_{5}, x_{3}, y_{2}, x_{1}\rangle, (x_{4}, x_{6}, x_{7}, y_{6}); \\
R_4: & \langle y_{1}, y_{3}, y_{2}, x_{2}, x_{4}, y_{4},  x_{3}, x_{1}\rangle, (x_{5}, x_{6}, y_{5}, y_{6}); \\
R_5: & \langle y_{0}, y_{2}, y_{1}, x_{1}, y_{3}, y_{5}, x_{7}, y_{7}, x_{5}, x_{4}, x_{3}, x_{2}, x_{0}\rangle, (y_{4}, x_{6}, x_{8}, y_{6}); \\
R_6: & \langle y_{1}, x_{3}, y_{4}, y_{5}, x_{4}, y_{2}, y_{3}, x_{2}, x_{1}, y_{0}\rangle, (x_{5}, y_{6}, y_{8}, x_{6}); \\
R_7: & \langle x_{1}, x_{3}, x_{5}, y_{7}, x_{7}, x_{6}, y_{4}, x_{4}, y_{6}, y_{5}, y_{3}, y_{1}\rangle, (x_{0}, y_{0}, x_{2}, y_{2}); \\
R_8: & \langle x_{0}, y_{2}, x_{4}, y_{5}, y_{7}, x_{6}, y_{8}, y_{6}, x_{7}, x_{5}, x_{3}, y_{1}, y_{0}\rangle, (x_{1}, x_{2}, y_{4}, y_{3}); \\
R_9: & \langle y_{0}, x_{1}, y_{2}, y_{4}, y_{6}, x_{6}, x_{4}, x_{2}, y_{1}\rangle, (x_{3}, y_{5}, x_{5}, y_{3}).
\end{array}
\]
\item $s_0=8:$
\[
\begin{array}{rl}
R_1: & \langle x_{0}, x_{1}, y_{1}, x_{2}, x_{3}, y_{3}, x_{4}, x_{5}, y_{5}, x_{6}, x_{8}, y_{6}, y_{4}, y_{2}, y_{0}\rangle, (x_{7}, y_{8}, y_{7}, x_{9}); \\
R_2: & \langle x_{1}, x_{0}, y_{1}, y_{2}, x_{3}, x_{4}, y_{3}, x_{5}, x_{6}, y_{5}, y_{4}, x_{2}, y_{0}\rangle, (y_{6}, x_{7}, y_{7}, y_{8}); \\
R_3: & \langle y_{0}, y_{1}, x_{0}, x_{2}, y_{2}, y_{3}, y_{4}, x_{4}, x_{6}, y_{6}, x_{5}, x_{3}, x_{1}\rangle, (y_{5}, x_{7}, x_{8}, y_{7}); \\
R_4: & \langle y_{1}, x_{3}, y_{4}, x_{5}, y_{6}, x_{4}, x_{2}, y_{3}, y_{2}, x_{1}\rangle, (y_{5}, y_{7}, x_{6}, x_{7}); \\
R_5: & \langle y_{0}, x_{2}, y_{4}, y_{5}, y_{6}, y_{8}, x_{7}, x_{9}, y_{7}, x_{8}, x_{6}, x_{5}, x_{4}, y_{2}, x_{0}\rangle, (x_{1}, y_{3}, x_{3}, y_{1}); \\
R_6: & \langle y_{1}, y_{3}, y_{5}, x_{5}, y_{4}, x_{6}, x_{4}, x_{3}, y_{2}, x_{2}, x_{1}, y_{0}\rangle, (y_{6}, y_{7}, y_{9}, x_{7}); \\
R_7: & \langle x_{1}, x_{3}, x_{5}, y_{7}, y_{6}, y_{5}, x_{4}, y_{4}, y_{3}, x_{2}, x_{0}, y_{0}, y_{2}, y_{1}\rangle, (x_{6}, y_{8}, x_{8}, x_{7}); \\
R_8: & \langle x_{0}, y_{2}, x_{4}, y_{6}, x_{8}, y_{8}, x_{6}, y_{4}, x_{3}, y_{5}, y_{3}, x_{1}, x_{2}, y_{1}, y_{0}\rangle, (x_{5}, x_{7}, y_{9}, y_{7}); \\
R_9: & \langle y_{0}, x_{1}, y_{2}, y_{4}, y_{6}, x_{6}, y_{7}, x_{7}, x_{5}, y_{3}, y_{1}\rangle, (x_{2}, x_{4}, y_{5}, x_{3}).
\end{array}
\]
\end{itemize}

\section{Admissible decompositions of $J_{2m}^*$ with small cycle lengths}
\label{Small 12-decomps}
We give the directed cycles an admissible $F$-decomposition of $J_{2m}^*$ for each 
\[F \in \{[2,4], [2,6], [2^2,4], [2^2,6], [2^2,4^2], [2^3], [2^3,4], [4^2], [4^3], [4,6], [4,8], [6]\}.
\]
These decompositions prove Lemma~\ref{SmallFactors}.

\subsection*{$F=[2,4]$}
\[
\begin{array}{ll}
F_1=\{(y_{1}, x_{3}), (x_{2}, y_{3}, y_{2}, x_{4})\} &F_6=\{(x_{2}, y_{4}), (x_{0}, x_{1}, y_{0}, y_{2})\}  \\
F_2=\{(x_{0}, y_{1}), (x_{1}, x_{2}, y_{2}, y_{3})\} & F_7=\{(y_{1}, x_{2}), (x_{1}, y_{3}, x_{3}, y_{2})\}\\
F_3=\{(y_{0}, y_{1}), (x_{1}, y_{2}, x_{2}, x_{3})\} & F_8=\{ (y_{2}, y_{4}), (x_{1}, x_{3}, y_{3}, x_{2})\}\\
F_4=\{(x_{0}, x_{2}), (y_{0}, x_{1}, y_{1}, y_{2})\} & F_9=\{ (y_{0}, x_{2}), (x_{0}, y_{2}, y_{1}, x_{1})\}\\
F_5=\{(y_{1}, y_{3}), (x_{2}, x_{4}, y_{2}, x_{3})\} & \\
\end{array}
\]

\subsection*{$F=[2,6]$}
\[
\begin{array}{cc}
F_1=\{(y_{1}, x_{2}), (y_{2}, x_{3}, x_{5}, y_{3}, x_{4}, y_{4})\} & F_6=\{(x_{3}, y_{5}), (x_{0}, x_{2}, y_{0}, x_{1}, y_{3}, y_{2})\}\\
F_2=\{(x_{0}, x_{1}), (y_{1}, y_{2}, x_{2}, x_{3}, y_{4}, y_{3})\} & F_7=\{(x_{1}, y_{1}), (x_{2}, y_{3}, y_{4}, x_{3}, y_{2}, x_{4})\}\\
F_3=\{(y_{1}, x_{3}), (y_{0}, x_{2}, x_{4}, y_{3}, x_{1}, y_{2})\} & F_8=\{(y_{3}, y_{5}), (x_{1}, x_{2}, y_{2}, y_{4}, x_{4}, x_{3})\} \\ 
F_4=\{(x_{0}, y_{1}), (y_{0}, y_{2}, y_{3}, x_{3}, x_{2}, x_{1})\} & F_9=\{(y_{0}, y_{1}), (x_{0}, y_{2}, x_{1}, x_{3}, y_{3}, x_{2})\}\\
F_5=\{(x_{2}, y_{4}), (y_{1}, y_{3}, x_{5}, x_{3}, x_{4}, y_{2})\} & \\
\end{array}
\]

\subsection*{$F=[2^2,4]$}
\[
\begin{array}{lclcl}
F_1=\{(y_{3}, x_{5}), (x_{4}, y_{4}), (y_{1}, x_{2}, y_{2}, x_{3})\} &F_6=\{(x_{1}, y_{3}), (x_{3}, y_{5}), (x_{0}, x_{2}, y_{0}, y_{2})\}\\
F_2=\{(x_{2}, y_{3}), (x_{3}, y_{4}), (x_{0}, x_{1}, y_{2}, y_{1}) \} & F_7=\{(y_{1}, y_{3}), (y_{2}, y_{4}), (x_{1}, x_{2}, x_{4}, x_{3})\}\\
F_3=\{(x_{1}, y_{1}), (y_{3}, x_{4}), (y_{0}, x_{2}, x_{3}, y_{2})\} &F_8=\{(x_{2}, y_{4}), (y_{3}, y_{5}), (x_{1}, x_{3}, x_{4}, y_{2})\} \\
F_4=\{ (y_{0}, x_{1}), (y_{2}, y_{3}), (x_{0}, y_{1}, x_{3}, x_{2})\} & F_9=\{(y_{0}, y_{1}), (x_{3}, y_{3}), (x_{0}, y_{2}, x_{2}, x_{1})\}\\
F_5=\{(x_{3}, x_{5}), (y_{3}, y_{4}), (y_{1}, y_{2}, x_{4}, x_{2})\} &  \\
\end{array}
\]

\subsection*{$F=[2^2,6]$}
\[
\begin{array}{lclcl}
F_1=\{(y_{1}, x_{2}), (y_{2}, x_{3}), (y_{3}, x_{4}, x_{6}, y_{4}, x_{5}, y_{5})\} &F_6=\{(x_{0}, x_{2}), (y_{4}, y_{6}), (y_{0}, x_{1}, y_{3}, x_{3}, x_{4}, y_{2})\}\\
F_2=\{(x_{0}, x_{1}), (y_{1}, y_{2}), (x_{2}, x_{3}, y_{3}, y_{4}, y_{5}, x_{4})\} & F_7=\{(x_{1}, y_{1}), (x_{2}, y_{4}), (y_{2}, x_{4}, y_{5}, x_{3}, x_{5}, y_{3})\}\\
F_3=\{(y_{0}, y_{1}), (x_{1}, x_{2}), (y_{2}, y_{3}, x_{5}, x_{4}, x_{3}, y_{4})\} &F_8=\{(x_{1}, y_{2}), (x_{4}, y_{6}), (x_{2}, y_{3}, y_{5}, x_{5}, y_{4}, x_{3})\} \\ 
F_4=\{(x_{0}, y_{1}), (x_{1}, x_{3}), (y_{0}, y_{2}, y_{4}, x_{4}, y_{3}, x_{2})\} & F_9=\{(x_{0}, y_{2}), (y_{1}, x_{3}), (y_{0}, x_{2}, x_{4}, y_{4}, y_{3}, x_{1})\}\\
F_5=\{(y_{1}, y_{3}), (x_{2}, y_{2}), (x_{3}, y_{5}, y_{4}, x_{6}, x_{4}, x_{5})\} &  \\
\end{array}
\]

\subsection*{$F=[2^2,4^2]$}

\[
\begin{array}{lclcl}
F_1=\{(y_{1}, x_{2}), (y_{2}, x_{3}), (y_{3}, x_{4}, x_{6}, y_{4}), (x_{5}, y_{6}, y_{5}, x_{7})\} \\
F_2=\{(x_{0}, x_{1}), (y_{1}, y_{2}), (x_{2}, x_{3}, y_{3}, y_{4}), (x_{4}, x_{5}, y_{5}, y_{6})\}\\
F_3=\{(y_{0}, x_{1}), (y_{1}, x_{3}), (x_{2}, x_{4}, y_{3}, y_{2}), (y_{4}, x_{6}, y_{5}, x_{5})\} \\
F_4=\{(x_{1}, x_{2}), (x_{3}, y_{4}), (x_{0}, y_{1}, y_{0}, y_{2}), (y_{3}, x_{5}, x_{4}, y_{5}) \} \\
F_5=\{(y_{1}, y_{3}), (y_{4}, y_{6}), (x_{2}, y_{2}, x_{4}, x_{3}), (x_{5}, x_{7}, y_{5}, x_{6}) \}  \\
F_6=\{(x_{1}, x_{3}), (x_{4}, y_{4}), (x_{0}, y_{2}, y_{0}, x_{2}), (y_{3}, y_{5}, y_{7}, x_{5})\}\\
 F_7=\{(x_{1}, y_{1}), (x_{3}, y_{5}), (x_{2}, y_{4}, y_{2}, y_{3}), (x_{4}, y_{6}, x_{5}, x_{6})\}\\
 F_8=\{(x_{1}, y_{2}), (x_{6}, y_{6}), (x_{2}, y_{3}, x_{3}, x_{4}), (y_{4}, x_{5}, y_{7}, y_{5})\} \\ 
 F_9=\{(x_{1}, y_{3}), (x_{3}, x_{5}), (x_{0}, x_{2}, y_{0}, y_{1}), (y_{2}, y_{4}, y_{5}, x_{4})\}\\
\end{array}
\]

\subsection*{$F=[2^3]$}

\[
\begin{array}{lclcl}
F_1=\{(y_{1}, x_{2}), (y_{2}, x_{4}), (x_{3}, y_{3})\} &F_6=\{(x_{0}, x_{1}), (y_{0}, y_{2}), (x_{2}, y_{4}) \}\\
F_2=\{(x_{0}, y_{1}), (x_{1}, x_{2}), (y_{2}, y_{3})\} & F_7=\{(x_{1}, y_{3}), (y_{1}, x_{3}), (x_{2}, y_{2})\}\\
F_3=\{(y_{0}, x_{1}), (y_{1}, y_{2}), (x_{2}, x_{3})\} &F_8=\{(x_{1}, x_{3}), (x_{2}, y_{3}), (y_{2}, y_{4})\} \\ 
F_4=\{(x_{0}, x_{2}), (y_{0}, y_{1}), (x_{1}, y_{2}) \} & F_9=\{(x_{0}, y_{2}), (y_{0}, x_{2}), (x_{1}, y_{1})\}\\
F_5=\{(y_{1}, y_{3}), (x_{2}, x_{4}), (y_{2}, x_{3})\} &  \\
\end{array}
\]

\subsection*{$F=[2^3,4]$}

\[
\begin{array}{lclcl}
F_1=\{(y_{1}, x_{2}), (y_{2}, x_{3}), (y_{3}, x_{5}), (x_{4}, y_{5}, y_{4}, x_{6})\} &F_6=\{(x_{0}, x_{2}), (y_{0}, y_{2}), (x_{4}, y_{6}), (x_{1}, y_{3}, y_{4}, x_{3}) \}\\
F_2=\{(x_{0}, x_{1}), (y_{1}, y_{2}), (x_{2}, x_{3}), (y_{3}, x_{4}, y_{4}, y_{5})\} & F_7=\{(x_{1}, y_{1}), (x_{2}, y_{4}), (x_{3}, x_{5}), (y_{2}, y_{3}, y_{5}, x_{4})\}\\
F_3=\{(y_{0}, x_{1}), (y_{1}, x_{3}), (x_{2}, y_{3}), (y_{2}, x_{4}, x_{5}, y_{4})\} &F_8=\{(x_{2}, x_{4}), (y_{4}, y_{6}), (x_{5}, y_{5}), (x_{1}, x_{3}, y_{3}, y_{2}) \} \\ 
F_4=\{(x_{0}, y_{1}), (y_{0}, x_{2}), (x_{3}, x_{4}), (x_{1}, y_{2}, y_{4}, y_{3}) \} & F_9=\{(x_{0}, y_{2}), (y_{0}, y_{1}), (x_{1}, x_{2}), (x_{3}, y_{4}, x_{4}, y_{3})\}\\
F_5=\{(y_{1}, y_{3}), (x_{2}, y_{2}), (x_{3}, y_{5}), (x_{4}, x_{6}, y_{4}, x_{5})\} &  \\
\end{array}
\]

\subsection*{$F=[4^2]$}

\[
\begin{array}{lclcl}
F_1=\{(y_{1}, x_{2}, x_{4}, y_{2}), (x_{3}, y_{4}, y_{3}, x_{5})\} &F_6=\{(x_{0}, x_{2}, y_{0}, y_{2}), (x_{1}, x_{3}, y_{5}, y_{3}) \}\\
F_2=\{(x_{0}, y_{2}, x_{2}, x_{1}), (y_{1}, y_{3}, y_{4}, x_{3})\} & F_7=\{(x_{1}, y_{3}, x_{2}, y_{1}), (y_{2}, y_{4}, x_{4}, x_{3})\}\\
F_3=\{(y_{0}, x_{1}, y_{1}, y_{2}), (x_{2}, x_{3}, y_{3}, x_{4})\} &F_8=\{(x_{1}, x_{2}, y_{4}, y_{2}), (x_{3}, x_{4}, y_{3}, y_{5}) \} \\ 
F_4=\{(x_{0}, x_{1}, y_{0}, y_{1}), (x_{2}, y_{3}, y_{2}, x_{3})  \} & F_9=\{(x_{0}, y_{1}, y_{0}, x_{2}), (x_{1}, y_{2}, y_{3}, x_{3})\}\\
F_5=\{(y_{1}, x_{3}, x_{5}, y_{3}), (x_{2}, y_{2}, x_{4}, y_{4}) \} &  \\
\end{array}
\]

\subsection*{$F=[4^3]$}

\[
\begin{array}{lclcl}
F_1=\{(y_{1}, x_{2}, y_{2}, x_{3}), (y_{3}, x_{4}, x_{6}, y_{4}), (x_{5}, y_{6}, y_{5}, x_{7})\} \\
F_2=\{ (x_{0}, x_{1}, y_{2}, y_{1}), (x_{2}, y_{3}, y_{4}, x_{4}), (x_{3}, y_{5}, y_{6}, x_{5})\}\\
F_3=\{(y_{0}, y_{1}, y_{3}, x_{1}), (x_{2}, x_{3}, x_{4}, y_{2}), (y_{4}, x_{6}, x_{5}, y_{5}) \} \\
F_4=\{(x_{0}, y_{1}, y_{0}, y_{2}), (x_{1}, x_{2}, y_{4}, x_{3}), (y_{3}, y_{5}, x_{5}, x_{4}) \} \\
F_5=\{(y_{1}, y_{2}, y_{3}, x_{2}), (x_{3}, x_{5}, x_{7}, y_{5}), (x_{4}, y_{4}, y_{6}, x_{6}) \}  \\
 F_6=\{(x_{0}, x_{2}, y_{0}, x_{1}), (y_{2}, x_{4}, x_{3}, y_{4}), (y_{3}, x_{5}, y_{7}, y_{5}) \} \\ 
 F_7=\{(x_{1}, y_{1}, x_{3}, x_{2}), (y_{2}, y_{4}, x_{5}, y_{3}), (x_{4}, y_{5}, x_{6}, y_{6}) \}\\
  F_8=\{(x_{1}, y_{3}, x_{3}, y_{2}), (x_{2}, x_{4}, y_{6}, y_{4}), (x_{5}, x_{6}, y_{5}, y_{7})\}\\
  F_9=\{(x_{0}, y_{2}, y_{0}, x_{2}), (x_{1}, x_{3}, y_{3}, y_{1}), (x_{4}, x_{5}, y_{4}, y_{5})\}\\
\end{array}
\]

\subsection*{$F=[4,6]$}

\[
\begin{array}{lclcl}
F_1=\{(y_{1}, x_{2}, y_{2}, x_{3}), (y_{3}, x_{4}, x_{6}, y_{4}, x_{5}, y_{5})\} & F_6=\{(x_{0}, x_{2}, y_{0}, x_{1}), (y_{2}, y_{3}, x_{3}, y_{4}, y_{6}, x_{4})\}\\
F_2=\{(x_{0}, x_{1}, y_{1}, y_{2}), (x_{2}, x_{3}, y_{3}, y_{4}, y_{5}, x_{4})\} & F_7=\{ (x_{3}, y_{5}, x_{5}, x_{4}), (x_{1}, x_{2}, y_{4}, y_{3}, y_{2}, y_{1})\} \\
F_3=\{(y_{0}, x_{2}, y_{3}, y_{1}), (x_{1}, x_{3}, x_{4}, x_{5}, y_{4}, y_{2})\} & 
F_8=\{(x_{3}, x_{5}, y_{3}, y_{5}), (x_{1}, y_{2}, x_{4}, y_{6}, y_{4}, x_{2})\} \\
F_4=\{(x_{0}, y_{1}, x_{3}, x_{2}), (y_{0}, y_{2}, y_{4}, x_{4}, y_{3}, x_{1})\} & F_9=\{(x_{0}, y_{2}, y_{0}, y_{1}), (x_{1}, y_{3}, x_{2}, x_{4}, y_{4}, x_{3})\} \\
F_5=\{(x_{4}, y_{5}, y_{4}, x_{6}), (y_{1}, y_{3}, x_{5}, x_{3}, y_{2}, x_{2})\} &\\
\end{array}
\]

\subsection*{$F=[4,8]$}

\[
\begin{array}{lclcl}
F_1=\{(y_{1}, x_{2}, y_{2}, x_{3}), (y_{3}, x_{4}, y_{4}, x_{6}, y_{6}, x_{5}, x_{7}, y_{5})\} \\
F_2=\{(x_{0}, x_{1}, y_{1}, y_{2}), (x_{2}, x_{3}, y_{3}, y_{4}, x_{5}, y_{5}, y_{6}, x_{4})\}\\
F_3=\{(y_{0}, x_{1}, x_{2}, y_{1}), (y_{2}, y_{3}, x_{5}, x_{4}, x_{6}, y_{5}, y_{4}, x_{3}) \} \\
F_4=\{(x_{0}, x_{2}, y_{0}, y_{1}), (x_{1}, x_{3}, y_{5}, x_{4}, x_{5}, y_{4}, y_{3}, y_{2}) \} \\
F_5=\{(y_{1}, x_{3}, x_{2}, y_{3}), (y_{2}, y_{4}, y_{6}, y_{5}, x_{7}, x_{5}, x_{6}, x_{4})  \}  \\
F_6=\{(y_{3}, y_{5}, y_{7}, x_{5}), (x_{0}, y_{2}, y_{0}, x_{2}, y_{4}, x_{4}, x_{3}, x_{1})\}\\
 F_7=\{(x_{1}, y_{2}, y_{1}, y_{3}), (x_{2}, x_{4}, y_{5}, x_{3}, x_{5}, y_{6}, x_{6}, y_{4}) \} \\ 
 F_8=\{(x_{5}, y_{7}, y_{5}, x_{6}), (x_{1}, y_{3}, x_{3}, x_{4}, y_{6}, y_{4}, y_{2}, x_{2})\}\\
 F_9=\{(x_{3}, y_{4}, y_{5}, x_{5}), (x_{0}, y_{1}, x_{1}, y_{0}, y_{2}, x_{4}, y_{3}, x_{2})\}\\
\end{array}
\]

\subsection*{$F=[6]$}

\[
\begin{array}{lclcl}
F_1=\{(y_{1}, x_{2}, x_{4}, y_{2}, x_{3}, y_{3})\} & F_4=\{(x_{0}, y_{1}, y_{0}, y_{2}, x_{1}, x_{2})\}& F_7=\{(x_{1}, x_{3}, y_{1}, y_{2}, y_{3}, x_{2})\} \\
 F_2=\{(x_{0}, x_{2}, y_{3}, x_{1}, y_{2}, y_{1})\}  & F_5=\{ (y_{1}, y_{3}, x_{3}, y_{2}, x_{4}, x_{2})\} & F_8=\{(x_{1}, y_{3}, y_{2}, y_{4}, x_{2}, x_{3}) \}\\ 
F_3=\{(y_{0}, x_{1}, y_{1}, x_{3}, x_{2}, y_{2})\}& F_6=\{(x_{0}, x_{1}, y_{0}, x_{2}, y_{4}, y_{2})\} & F_9=\{(x_{0}, y_{2}, x_{2}, y_{0}, y_{1}, x_{1})\} \\
\end{array}
\]

\end{document}